\documentclass[11pt,a4paper]{article}  

\usepackage[paper=a4paper, left=18mm, top=24mm, right=18mm, bottom=26mm]{geometry}

\usepackage{color}              
\definecolor{usablegreen}{rgb}{0,.6,0}
\definecolor{usablecyan}{rgb}{0,.6,.6}
\usepackage{graphicx}
\usepackage{amsmath}
\usepackage{amssymb}
\usepackage[linktocpage=true]{hyperref}
\usepackage{latexsym, enumerate, amsfonts, amsthm, bbm}
\usepackage{dsfont}
\usepackage[small]{caption}

\newtheorem{Lemma}{Lemma}[section]
\newtheorem{lemma}[Lemma]{Lemma}

\newtheorem{prop}[Lemma]{Proposition}

\newtheorem{theorem}[Lemma]{Theorem}
\newtheorem*{LusinsThm}{Lusin's theorem}

\newtheorem{corollary}[Lemma]{Corollary}
\newtheorem{Definition}[Lemma]{Definition}
\newtheorem{definition}[Lemma]{Definition}

\newtheorem{Notation}[Lemma]{Notation}
\theoremstyle{definition}
\newtheorem{Example}[Lemma]{Example}
\newtheorem{Remark}[Lemma]{Remark}
\newtheorem{remark}[Lemma]{Remark}
\newenvironment{enremark}{\begin{Remark}\begin{enumerate}}{\end{enumerate}\end{Remark}}

\newenvironment{enproof}{\begin{proof}\begin{enumerate}}{\qedhere\end{enumerate}\end{proof}}

\newenvironment{acknowledgements}{\subsection*{Acknowledgements}}{}

\newenvironment{example}[1][]{\begin{Example}[#1]}{\end{Example}}

\newcommand{\lemref}[1]{Lemma~\textup{\ref{lem:#1}}}
\newcommand{\propref}[1]{Proposition~\textup{\ref{p:#1}}}
\newcommand{\corref}[1]{Corollary~\textup{\ref{c:#1}}}
\newcommand{\thmref}[1]{Theorem~\textup{\ref{t:#1}}}
\newcommand{\remref}[1]{Remark~\textup{\ref{r:#1}}}
\newcommand{\secref}[1]{Section~\ref{sec:#1}}
\newcommand{\subref}[1]{Subsection~\ref{sub:#1}}
\newcommand{\exref}[1]{Example~\ref{ex:#1}}
\newcommand{\itref}[1]{\ref{it:#1}}
\newcommand{\defref}[1]{Definition~\ref{d:#1}}

\newcommand{\D}{\mathcal{D}}
\newcommand{\DE}[1][J]{\D_E(#1)}
\newcommand{\EE}{\mathbb{E}}
\newcommand{\MM}{\mathbb{M}}
\newcommand{\MMI}{\MM_I}		
\newcommand{\FMI}{\MMI^\mathrm{fct}}	
\newcommand{\Mh}[1][h]{\MMI^{#1}}
\newcommand{\Mheps}{\Mh[h_\eps]}
\newcommand{\Mde}[1][\delta]{\Mh[#1,\eps]}
\newcommand{\Mdmh}[1][h_\eps]{\Mh[\delta_m,#1(\delta_m)]}

\newcommand{\Mdefull}[2][\delta]{\Mh[#1,#2]}
\newcommand{\Mdh}[1][\delta]{\Mh[#1,h(#1)]}
\newcommand{\Mdheps}[1][h_\eps]{\Mh[\delta,#1(\delta)]}
\newcommand{\HMI}[1][h]{\mathfrak{M}_I^{#1}}
\newcommand{\DEMI}[1][\delta]{\HMI[#1,\eps]}
\newcommand{\DEMIfull}[2]{\HMI[#1,#2]}
\newcommand{\Ade}[1][X]{A_{\delta,\eps}^{#1}}
\newcommand{\Adtwoe}[1][X]{A_{2\delta,\eps}^{#1}}
\newcommand{\Adh}[1][X]{A_{\delta,h(\delta)}^{#1}}
\newcommand{\NN}{\mathbb{N}}
\newcommand{\N}{\NN}
\newcommand{\PP}{\mathbb{P}}
\renewcommand{\P}{\PP}
\newcommand{\Ps}[1]{\P(\{#1\})}
\newcommand{\bPs}[1]{\P\(\bigl\{\,#1\,\bigr\}\)}
\newcommand{\RR}{\mathbb{R}}
\newcommand{\R}{\RR}

\newcommand{\B}{\mathfrak{B}}

\newcommand{\BMMI}[1][\delta]{B_{#1}^{\MMI}}	

\newcommand{\CH}{\mathcal{H}}

\newcommand{\CM}{\mathcal{M}}

\newcommand{\1}{\mathds{1}}

\DeclareMathOperator{\supp}{supp}
\DeclareMathOperator{\diam}{diam}
\DeclareMathOperator{\cont}{cont}

\renewcommand{\epsilon}{\varepsilon}\newcommand{\eps}{\varepsilon}
\newcommand{\X}{\mathcal{X}}
\newcommand{\smallcal}[1]{
	{\mathchoice{\scriptstyle}{\scriptstyle}{\scriptscriptstyle}{\scriptscriptstyle}\mathcal{#1}}}
\newcommand{\smallx}{\smallcal{X}}
\newcommand{\smally}{\smallcal{Y}}
\newcommand{\distmatnu}{\boldsymbol{\nu}}
\newcommand{\musq}[1][\mu]{{#1}^{\otimes 2}}
\newcommand{\mupsq}[1][\mu]{(#1')^{\otimes 2}}
\newcommand{\munpsq}[1][n]{\mupsq[\mu_{#1}]}

\newcommand{\muh}{\hat{\mu}}
\newcommand{\hh}{\hat{h}}
\newcommand{\Xh}{\hat{X}}
\newcommand{\rh}{\hat{r}}

\newcommand{\xh}{\hat{x}}
\newcommand{\uh}{\hat{u}}

\newcommand{\smallxh}{\,\hat{\!\smallx}}

\newcommand{\Betrag}[1]{\Bigl| #1 \Bigr|}

\newcommand{\wtspace}{\mathchoice{\,}{\,}{}{}}			
\newcommand{\wmspace}{\mathchoice{\;}{\,}{}{}}			

\newcommand{\bigmid}{\bigm|}
\newcommand{\bsetbar}[2]{\bigl\{\, #1 \bigmid #2 \,\bigr\}}
\newcommand{\set}[2]{\{\wtspace #1 : #2 \wtspace\}}		
\newcommand{\bset}[2]{\bigl\{#1 : #2\bigr\}}

\newcommand{\tomGw}{\xrightarrow{\mathrm{mGw}}}
\newcommand{\tonmGw}{\xrightarrow[\scriptscriptstyle n\to\infty]{\mathrm{mGw}}}
\newcommand{\tow}[1][]{\xrightarrow[\scriptscriptstyle #1]{w}}

\newcommand{\nlim}{\lim_{n\to\infty}}
\newcommand{\nliminf}{\liminf_{n\to\infty}}

\newcommand{\deltalimsup}{\limsup_{\delta\downarrow0}}
\newcommand{\nlimsup}{\limsup_{n\to\infty}}
\newcommand{\restricted}[1]{{\mathclose|}_{#1}}
\newcommand{\bigldelimiter}[1]{\mathchoice{\bigl#1}{\bigl#1}{{\textstyle#1}}{{\scriptstyle#1}}}
\newcommand{\bigrdelimiter}[1]{\mathchoice{\bigr#1}{\bigr#1}{{\textstyle#1}}{{\scriptstyle#1}}}
\renewcommand{\(}{\bigldelimiter(}
\renewcommand{\)}{\bigrdelimiter)}

\newcommand{\folge}[2][n]{(#2_{#1})_{#1\in\N}}

\newcommand{\ton}[1][n]{\wmspace\displaystyle\mathop{\longrightarrow}_{\scriptscriptstyle #1\to\infty}\wmspace}	

\newcommand{\dmGP}{d_\mathrm{mGP}}
\newcommand{\dfmi}{d_\mathrm{fGP}}
\newcommand{\dPr}[1][]{d_\mathrm{Pr}^{#1}}

\newcommand{\Mf}{\CM_\mathrm{f}}	

\newcommand{\floor}[1]{\lfloor#1\rfloor}
\renewcommand{\d}{\mathrm{d}}	
\newcommand{\dx}{\d x}
\newcommand{\dy}{\d y}
\newcommand{\dt}{\d t}
\newcommand{\du}{\d u}
\newcommand{\integralspace}{\/\mathchoice{\;}{\,}{\,}{}}		
	
\newcommand{\probinta}[4]{\int_{#1} #2 \integralspace #3(\d#4)}

\newcommand{\probintamu}[3]{\int_{#1} #2 \integralspace \mu\(\d#3\)}
\newcommand{\probintanu}[3]{\int_{#1} #2 \integralspace \nu(\d#3)}

\newcommand{\define}[1]{\emph{#1}}
\newcommand{\comment}[1]{}
\newenvironment{proofsteps}{\setcounter{enumi}{0}}{}
\newcommand{\step}{\refstepcounter{enumi}\removelastskip\smallskip\par\noindent\emph{Step \arabic{enumi}.} \hspace{0.5ex}}
\newcommand{\proofcase}[3][\Rightarrow]{\removelastskip\smallskip\par\noindent\emph{``#2$\,#1\,$#3'':}\hspace{0.5ex}}
\newcommand{\eproofcase}[2]{\proofcase[\Leftrightarrow]{#1}{#2}}

\newcommand{\cadlag}{c\`adl\`ag}
\newcommand{\modcadlag}{modulus of \cadlag ness}
\newcommand{\modscadlag}{moduli of \cadlag ness}

\newcommand{\lbeq}[1]{\label{#1}}

\newcommand{\nn}{\nonumber}

\newcommand{\picturefig}[4]{
\begin{figure}[t]
\begin{center}
\includegraphics[scale=#1]{#2}
\end{center}
\caption{#3}
\label{pic:#4}
\end{figure}}

\numberwithin{equation}{section}

\begin{document}

\author{
Sandra Kliem%
\thanks{
Fakult\"at f\"ur Mathematik, Universit\"at Duisburg-Essen, Thea-Leymann-Str.\ 9, D-45127 Essen,
Germany.\newline
\hspace*{1.8em}E-mail: {\tt sandra.kliem@uni-due.de}, {\tt wolfgang.loehr@uni-due.de}}
\and
Wolfgang L\"ohr$^{*}$
}

\title{Existence of mark functions in marked metric measure spaces\footnotetext[0]{Preprint of \emph{Electronic
Journal of Probability}, \textbf{20}, no.\ 73, pp.\ 1--24\\[-2.1ex]}}
\date{\today}

\maketitle 

\begin{abstract}
We give criteria on the existence of a so-called mark function in the context of marked metric measure spaces
(mmm-spaces). If an mmm-space admits a mark function, we call it functionally-marked metric measure space
(fmm-space). This is not a closed property in the usual marked Gromov-weak topology, and thus we put particular
emphasis on the question under which conditions it carries over to a limit.
We obtain criteria for deterministic mmm-spaces as well as random mmm-spaces and mmm-space-valued processes.
As an example, our criteria are applied to prove that the tree-valued Fleming-Viot dynamics with mutation and selection from
\cite{DGP12} admits a mark function at all times, almost surely. Thereby, we fill a gap in a former proof of this
fact, which used a wrong criterion. 

Furthermore, the subspace of fmm-spaces, which is dense and not closed, is investigated in detail. We show that
there exists a metric that induces the marked Gromov-weak topology on this subspace and is complete. Therefore,
the space of fmm-spaces is a Polish space. We also construct a decomposition into closed sets which are related
to the case of uniformly equicontinuous mark functions.  
\smallskip

\noindent
{\bf Key words:} mark function; tree-valued Fleming-Viot process; mutation; marked metric measure space; Gromov-weak topology; Prohorov metric; Lusin's theorem.

\smallskip
\noindent
{\bf MSC2000 subject classification.} {Primary
60K35, 
Secondary
60J25, 
60G17, 
60G57. 
} 

\end{abstract}

\smallskip
\begin{quote}\begin{quote}\begin{quote}
\def\section{\subsection}
\footnotesize
\tableofcontents
\end{quote}\end{quote}\end{quote}
\smallskip

\section{Introduction}

%

A metric (finite) measure spaces (\emph{mm-space}) is a complete, separable metric space $(X,r)$ together with a finite
measure $\nu$ on it. Considering the space of (equivalence classes of) mm-spaces itself as a metric space dates
back to Gromov's invention of the $\underline\Box_\lambda$-metric in \cite[Chapter~3$\frac12$]{Gromov}.
Motivated by Aldous' work on the Brownian continuum random tree (\cite{Aldous:CRT3}), it was realised in
\cite{GPW09} that the space of mm-spaces is a useful state space for tree-valued stochastic processes, and
Polish when equipped with the Gromov-weak topology. That the Gromov-weak topology actually coincides with the
one induced by the $\underline\Box_\lambda$-metric was shown in \cite{Loehr13}.

Important examples for the use of mm-spaces within probability theory are individual-based populations $X$ with
given mutual genealogical distances $r$ between individuals. Here, $r$ can for instance measure the time to the most recent
common ancestor (MRCA) (cf.\ \cite[(2.7), Remark~3.3]{DGP12}), where the resulting metric space is ultrametric.
Another possibility is the number of mutations back to the MRCA (cf.\ \cite{KW2014}), where the resulting space is not
ultrametric. Finally, there is a sampling measure $\nu$ on the space $X$ which models population density.
This means that the state of the process is an mm-space $(X,r,\nu)$.
Such individual-based models are often formulated for infinite population size (with diffuse measures $\nu$) but
obtained as the high-density limit of approximating models with finite populations (where $\nu$ is typically the
uniform distribution on all individuals). 

For encoding more information about the individuals, such as an (allelic) type or location (which may change over time),
marked metric measure spaces (mmm-spaces) and the corresponding marked Gromov-weak topology (mGw-topology) have
been introduced in \cite{DGP11}.
For a fixed complete, separable metric space $(I,d)$ of marks, the sampling measure $\nu$ is replaced by a measure
$\mu$ on $X\times I$, which models population density in combination with mark distribution.

A natural question in this context is whether or not every point of the limiting population $X$ has a single
mark almost surely, that is, does genetic distance zero imply the same type/location?
Put differently, we ask ourselves if $\mu$ factorizes into a ``population density'' measure $\nu$ on $X$
and a mark function $\kappa\colon X\to I$ assigning each individual its mark. If this is the case, we call the
mmm-space functionally-marked (fmm-space). This property is often desirable, and one might want to consider the
space of fmm-spaces, rather than mmm-spaces, as the state space. Unfortunately, the subspace of fmm-spaces is not
closed in the mGw-topology, which means that limits of finite-population models that are constructed as fmm-spaces might
not admit mark functions themselves. It is therefore of interest, if the space of fmm-spaces with marked Gromov-weak topology
is a Polish space (that is a ``good'' state space). Here, we show in \thmref{Polish} that this is indeed the
case. We also produce criteria to enable one to check if an mmm-space admits a mark function.
For limiting populations, they are given in terms of the approximating mmm-spaces. We derive such criteria
for deterministic spaces (\thmref{modulus}), random spaces (\thmref{rnd-modulus}) and mmm-space-valued
processes (\thmref{pr-modulus} and \thmref{modcadlag}). 

An important example of such a high-density limit of approximating models with finite populations is the
tree-valued Fleming-Viot dynamics. In the neutral case, it is constructed in \cite{GPW13} using the formalism of
mm-spaces. In \cite{DGP12}, (allelic) types -- encoded as marks of mmm-spaces -- are included, in order
to model mutation and selection.
For this process, the question of existence of a mark function has already been posed.
In \cite[Remark~3.11]{DGP12} and \cite[Theorem~6]{DGP13} it is stated that the tree-valued Fleming-Viot process
admits a mark function at all times, almost surely.
The given proof, however, contains a gap, because it relies on the criterion claimed in \cite[Lemma~7.1]{DGP13}, which is
wrong in general, as we show in \exref{counter}.
We fill this gap by applying our criteria and showing in \thmref{mark-FV} that the claim is indeed true and the
tree-valued Fleming-Viot process with mutation and selection (TFVMS) admits a mark function at all times, almost
surely. We also show in \thmref{mark-Lambda} that the same arguments apply to the $\Lambda$-version of the TFVMS
in the neutral case, that is where selection is not present.

Intuitively, the existence of a mark function in the case of the TFVMS holds because mutations are large but rare in the approximating sequence of tree-valued Moran models.
Hence, as genealogical distance becomes small, the probability that any mutation happened at all in the
close past becomes small as well (recall that distance equals time to the MRCA). In contrast, in \cite{KW2014}, where evolving phylogenies of trait-dependent branching with mutation and competition are under investigation, mutations happen at a high rate but are
small which justifies the hope for the existence of a mark function also for the limiting model. Our criteria
are also suited for this kind of situation.
%

\medskip\noindent\textbf{Outline. }
The paper is organized as follows. In the subsections of the introduction we first introduce notations and basic
results for the Prohorov metric for finite measures. Then, we give a short introduction to the space $\MMI$ of
marked metric measure spaces (mmm-spaces) with the marked Gromov-weak topology, as well as the marked
Gromov-Prohorov metric $\dmGP$ on it. We continue with defining the so-called functionally-marked metric measure
spaces (fmm-spaces) $\FMI\subseteq \MMI$, and finally investigate the case of equicontinuous mark functions as an
illustrative example.
We emphasize that the restriction of the marked Gromov-Prohorov metric $\dmGP$ to $\FMI$ is not complete. 

In \secref{Polish}, we therefore show that there exists another metric on $\FMI$ that induces the marked
Gromov-weak topology and is complete. As one sees in \subref{equicont}, the situation becomes easy if we
restrict to a subspace of $\MMI$ containing spaces with uniformly equicontinuous mark functions. We introduce in
\subref{betaestim} several related subspaces capturing some aspect of equicontinuity, and obtain a decomposition
of $\FMI$ into closed sets. This decomposition is used to prove Polishness of $\FMI$, and in \secref{criteria} to
formulate criteria for the existence of mark functions.

\secref{criteria} gives criteria for the existence of mark functions. Based on the construction of the
complete metric and the decomposition of $\FMI$, we derive in \subref{criteria-1} criteria to
check if an mmm-space admits a mark function, especially in the case where it is given as a
marked Gromov-weak limit.
We then transfer the results in \subref{criteria-2} to random mmm-spaces and in \subref{criteria-3} to
$\MMI$-valued stochastic processes. 

To conclude, \secref{examples} gives examples. We first show that the criterion in \cite{DGP13} is wrong in
general by means of counterexamples. Our criteria are then applied in \subref{FV} to prove the existence of a
mark function for the tree-valued Fleming-Viot dynamics with mutation and selection. To this goal, we verify the
necessary assumptions for a sequence of approximating tree-valued Moran models. In \subref{TFV} we show that a
similar strategy applies if we replace the tree-valued Moran models by so-called tree-valued $\Lambda$-Cannings
models. Finally, in \subref{FApp}, a future application to evolving phylogenies of trait-dependent branching with
mutation and competition is indicated.

\subsection{Notations and prerequisites} 

In this paper, let all topological spaces be equipped with their Borel $\sigma$-algebras.
We use the following notation throughout the article. 

\begin{Notation}
For a Polish space $E$, let $\mathcal{M}_1(E)$
respectively $\Mf(E)$ denote the space of probability respectively finite measures on the Borel $\sigma$-algebra
$\B(E)$ on $E$. The space $\Mf(E)$ is always equipped with the topology of weak convergence, which is denoted by $\tow$.
We also use the distance in variational norm of\/ $\mu,\nu \in \Mf(E)$, which is
\begin{equation}
	\| \mu-\nu \| := \sup_{B\in\B(E)} \bigl|\mu(B)-\nu(B)\bigr|.
\end{equation}
In particular, $\| \mu \| = \mu(E)$, and\/ $\|\mu-\nu\| = \nu(E)-\mu(E) $ if\/ $\mu \leq \nu$, that is $\mu(A)
\leq \nu(A)$ for all\/ $A \in \B(E)$. 

For\/ $Y \in \B(E)$ and\/ $\mu \in \Mf(E)$, denote by\/ $\mu \restricted Y\in \Mf(E)$ the restriction
of\/ $\mu$ to\/ $Y$, that is\/ $\mu\restricted Y(B):=\mu(B \cap Y)$ for all\/ $B \in \B(E)$. Because\/
$\mu \restricted Y \leq \mu$, we have\/ $\| \mu \restricted Y - \mu \| = \mu(E\setminus Y)$.

For $\varphi\colon E \rightarrow F$ measurable, with $F$ some other Polish space, denote the image measure of $\mu$
under $\varphi$ by $\varphi_* \mu:=\mu\circ \varphi^{-1}$.
Finally, for the product space $X:= E \times F$, the canonical projection operators from $X$ onto $E$ and $F$
are denoted by $\pi_E$ and $\pi_F$, respectively.
\end{Notation}

\begin{definition}[Prohorov metric]
	For finite measures\/ $\mu_0, \mu_1$ on a metric space\/ $(E, r)$, the \define{Prohorov metric} is
	defined as
	\begin{equation}
		\dPr(\mu_0, \mu_1) := \inf\bset{\eps>0}{\mu_i(A) \le \mu_{1-i}(A^\eps) + \eps \;\;\forall A\in
			\B(E),\, i\in\{0,1\}},
	\end{equation}
	where\/ $A^\eps := \set{x\in E}{r(A, x) < \eps}$ is the $\eps$-neighbourhood of\/ $A$.
\end{definition}

It is well-known that the Prohorov metric metrizes the weak convergence of measures if and only if the
underlying metric space is separable. The following equivalent expression for the Prohorov metric turns out to be useful.

\begin{remark}[coupling representation of the Prohorov metric]
	Let $(E,r)$ be a separable metric space and $\mu_1,\mu_2\in \CM_1(E)$.
	For a finite measure $\xi$ on $E^2$, we denote the marginals as $\xi_1 := \xi(\cdot \times E)$ and
	$\xi_2:=\xi(E\times \cdot)$.  It is well-known (see, e.g., \cite[Theorem~III.1.2]{EK}) that
	\begin{equation}\label{eq:Prc}
		\dPr(\mu_1, \mu_2) = \inf\bset{\eps>0}{\exists \xi\in\CM_1(E^2) \text{ with }
			\xi(N_\eps) \le \eps,\;\xi_i=\mu_i,\,i=1,2}, 
	\end{equation}
	where $N_\eps := \set{(x,y)\in E^2}{r(x,y) \ge \eps}$. We obtain from this equation
	\begin{equation}\label{eq:Prcouple}
		\dPr(\mu_1,\mu_2) = \inf\bset{\eps>0}{\exists \xi'\in\Mf(E^2) \text{ with } \xi'(N_\eps)=0,\;
			\xi_i' \le \mu_i,\, \|\mu_i - \xi_i'\| \le \eps, \,i=1,2}.
	\end{equation}
	Indeed, consider $\xi':=\xi\restricted{E^2\setminus
	N_\eps}$ respectively $\xi:=\xi'+(1-\|\xi'\|)^{-1} \big( (\mu_1-\xi'_1) \otimes (\mu_2-\xi'_2) \big)$ to obtain equality in the above.
	Following the ideas of the proof of the representation \eqref{eq:Prc} in \cite{EK}, the representation \eqref{eq:Prcouple} for the Prohorov metric $\dPr(\mu_0, \mu_1)$ is easily seen to hold true for
	measures $\mu_1,\mu_2\in\Mf(E)$ as well, which are not necessarily probability measures.
\end{remark}

From \eqref{eq:Prcouple}, we can easily deduce the following lemma, which we use below.
\begin{lemma}[rectangular lemma]\label{lem:rectangle}
	Let\/ $(E,r)$ be a separable, metric space, $\eps,\delta>0$, and\/ $\mu_1,\mu_2\in\Mf(E)$. Assume that\/
	$\dPr(\mu_1,\mu_2) < \delta$ and there is\/ $\mu_1'\le \mu_1$ with\/ $\|\mu_1-\mu_1'\| \le \eps$. Then
	\begin{equation}\label{eq:rectangle}
		\exists \mu_2' \le \mu_2: \dPr(\mu_1', \mu_2') < \delta,\; \|\mu_2-\mu_2'\| \le \eps. 
	\end{equation}
\end{lemma}
\begin{proof}
	According to \eqref{eq:Prcouple}, we find $\xi\in\Mf(E^2)$ with marginals $\xi_i \le \mu_i$, $i=1,2$,
	$\|\mu_i-\xi_i\| <\delta$, and $\xi(\{r\ge\delta\})=0$.
	Let $L$ be a probability kernel from $E$ to $E$ (for existence see \cite[Theorems~8.36--8.38]{Kle14}) with $\xi=\mu_1\otimes L$ and define
	$\xi':=(\mu_1' \land \xi_1)\otimes L$.
	Obviously, $\xi'_1\le \mu_1'$ and $\|\mu_1'-\xi_1'\| \le \|\mu_1-\xi_1\| < \delta$. Now set
	\begin{equation}
		\mu_2' := \xi'_2 + \mu_2 - \xi_2.
	\end{equation}
	Then $\xi_2'\le\mu_2'$, $\|\mu_2' - \xi'_2\| = \|\mu_2-\xi_2\| < \delta$ and thus $\dPr(\mu_2,\mu_2') <
	\delta$ by \eqref{eq:Prcouple}.
	Furthermore, $\mu_2'\le \mu_2$ and $\|\mu_2-\mu_2'\| = \|\xi_2 - \xi_2'\| \le \|\mu_1 - \mu_1'\| \le
	\eps$.
\end{proof}

\subsection{The space of marked metric measure spaces (mmm-spaces)}

In this subsection, we recall the space $\MMI$ of marked metric measure spaces, and the marked Gromov-Prohorov
metric $\dmGP$, which induces the marked Gromov-weak topology on it. This space, $(\MMI, \dmGP)$, will be the
basic space used in the rest of the paper. These concepts have been introduced in \cite{DGP11}, and are based on
the corresponding non-marked versions introduced in \cite{GPW09}. In contrast to \cite{DGP11}, we allow the
measures of the marked metric measure spaces to be finite, that is do not restrict ourselves to probability measures only. Because a
sequence of finite measures converges weakly if and only if their total masses and the normalized measures converge, or
the masses converge to zero, this straight-forward generalization requires only minor modifications (compare
\cite[Section~2.1]{LoehrVoisinWinter14}, where this generalization is done for metric measure spaces without
marks).

In what follows, fix a complete, separable metric space $(I,d)$, called the \emph{mark space}. It is the same for all marked metric
measure spaces in $\MMI$.
%
\begin{definition}[mmm-spaces, $\MMI$]
\begin{enumerate}
\item 
An \emph{($I$-)marked metric measure space (mmm-space)} is a triple $(X,r,\mu)$ such that\/ $(X,r)$ is a complete,
separable metric space, and\/ $\mu \in \Mf(X \times I)$, where $X \times I$ is equipped with the product topology.
\item
Let\/ $\smallx_i = (X_i,r_i,\mu_i)$, $i=1,2$, be two mmm-spaces, and\/ $\nu_i:=\mu_i(\cdot \times I)$ the
marginal of\/ $\mu_i$ on $X_i$. For a map $\varphi\colon X_1 \to X_2$ we use the notation
\begin{equation}\label{eq:tildephi}
	\tilde\varphi \colon X_1 \times I \to X_2 \times I, \quad (x,u) \mapsto \tilde{\varphi}(x,u) := (\varphi(x),u).
\end{equation}
We call\/ $\smallx_1$ and\/ $\smallx_2$ \emph{equivalent} if they are measure- and mark-preserving
isometric, that is there is an isometry $\varphi\colon \supp(\nu_1) \rightarrow \supp(\nu_2)$, such that
\begin{equation}
	\tilde{\varphi}_* \mu_1 = \mu_2.
\end{equation}
\item Finally, define
\begin{equation}
	\MMI := \bigl\{ \text{equivalence classes of mmm-spaces\/} \bigr\}.
\end{equation}
With a slight abuse of notation, we identify an mmm-space with its equivalence class and write
$\smallx = (X,r,\mu) \in \MMI$ for both mmm-spaces and equivalence classes thereof.
\end{enumerate}
\end{definition}

Next, we recall the marked Gromov-weak topology from \cite[Section~2.2]{DGP11} that turns $\MMI$ into a Polish
space (cf.\ \cite[Theorem~2]{DGP11}). To this goal, we first recall
\begin{definition}[marked distance matrix distribution]
Let\/ $\smallx := (X,r,\mu) \in \MMI$ and
\begin{equation}
  R^{(X,r)} := \left\{\begin{matrix}
                      (X \times I)^{\mathbb{N}} & \rightarrow & \R_+^{\binom\N2} \times I^{\mathbb{N}}, \cr
                      \big( (x_k,u_k)_{k \geq 1} \big) & \mapsto &\big( \big(r(x_k,x_l) \big)_{1 \leq k < l}, (u_k)_{k \geq 1} \big). \cr
                    \end{matrix} \right.
\end{equation}
The \emph{marked distance matrix distribution} of\/ $\smallx$ is defined as
\begin{equation}
  \distmatnu^\smallx := \|\mu\| \cdot \big( R^{(X,r)} \big)_* (\tfrac\mu{\|\mu\|})^{\mathbb{N}}
  	\in \Mf\big(\R_+^{\binom{\N}{2}} \times I^{\mathbb{N}} \big).
\end{equation}
\end{definition}
The marked Gromov-weak topology is the one induced by the map $\smallx\mapsto \distmatnu^\smallx$.
\begin{definition}[marked Gromov-weak topology]\label{def:gw-top}
Let\/ $\smallx, \smallx_1,\smallx_2,\ldots \in \MMI$. We say that\/ $(\smallx_n)_{n\in\N}$ converges to $\smallx$ in
the \emph{marked Gromov-weak topology}, $\smallx_n \tonmGw \smallx$, if and only if
\begin{equation}
  \distmatnu^{\smallx_n} \tow[n\to\infty] \distmatnu^\smallx
\end{equation}
in the weak topology on $\Mf\big( \mathbb{R}_+^{\binom{\N}{2}} \times I^{\mathbb{N}} \big)$.
\end{definition}

Finally, let us recall the Gromov-Prohorov metric from \cite[Section~3.2]{DGP11}. It is complete and metrizes
the marked Gromov-weak topology, as shown in \cite[Proposition~3.7]{DGP11}.
\begin{definition}[marked Gromov-Prohorov metric, $\dmGP$]
For $\smallx_i=(X_i,r_i,\mu_i) \in \MMI, i=1,2$, set
\begin{equation}
  \dmGP(\smallx_1,\smallx_2) :=
  	\inf_{(E,\varphi_1,\varphi_2)} \dPr\big( (\tilde{\varphi}_1)_* \mu_1, (\tilde{\varphi}_2)_* \mu_2 \big),
\end{equation}
where the infimum is taken over all complete, separable metric spaces $(E,r)$ and isometric embeddings
$\varphi_i\colon X_i \rightarrow E$, and\/ $\tilde\varphi_i$ is as in \eqref{eq:tildephi}, $i=1,2$. The 
Prohorov metric $\dPr$ is the one on $\Mf(E \times I)$, based on the metric $\tilde{r} = r + d$ on $E \times I$,
metrizing the product topology. The metric $\dmGP$ is called the \emph{marked Gromov-Prohorov metric}.
\end{definition}

A direct consequence of the fact that $\dmGP$ induces the marked Gromov-weak topology is the following
characterization of marked Gromov-weak convergence obtained in \cite[Lemma~3.4]{DGP11}.

\begin{lemma}[embedding of marked Gromov-weakly converging sequences]\label{lem:seqembed}
	Let\/ $\smallx_n=(X_n,r_n,\mu_n) \in \MMI$ for $n\in\N\cup\{\infty\}$. Then\/ $(\smallx_n)_{n\in\N}$
	converges to $\smallx_\infty$ Gromov-weakly if and only if there is a complete, separable metric space\/
	$(E,r)$, and isometric embeddings\/ $\varphi_n \colon X_n \to E$, such that for $\tilde\varphi_n$ as in
	\eqref{eq:tildephi},
	\begin{equation}
		(\tilde\varphi_n)_*\mu_n \tow \mu.
	\end{equation}
\end{lemma}

\subsection{Functionally-marked metric measure spaces (fmm-spaces)} 

Consider an $I$-marked metric measure space $\smallx=(X,r,\mu) \in \MMI$. Since $\mu$ is a finite measure on the
Polish space $X \times I$, regular conditional measures exist (cf.\ \cite[Theorems~8.36--8.38]{Kle14}), and we write 
\begin{equation}\label{nuK}
  \mu(\dx,\du) 
  = \nu(\dx) \cdot K_x(\du), 
\end{equation}
in short $\mu=\nu\otimes K$, for the marginal $\nu := \mu(\cdot \times I) \in \Mf(X)$, and a ($\nu$-a.s.\
unique) probability kernel $K$ from $X$ to $I$.

In the present article we investigate criteria for the \emph{existence of a mark function} for $\smallx$, that
is (cf.\ \cite[Section 3.3]{DGP13}) a measurable function $\kappa\colon X \to I$ such that
\begin{equation}
\lbeq{mark-fcn}
  \mu(\dx,\du) = \nu(\dx) \cdot \delta_{\kappa(x)}(\du),
\end{equation}
or equivalently, $K_x = \delta_{\kappa(x)}$ for $\nu$-almost every $x$.
Obviously, $\smallx$ admits a mark function if and only if $K_x$ is a Dirac measure for $\nu$-almost every $x$.
Recall that the complete, separable mark space $(I,d)$ is fixed once and for all.

\begin{Definition}[fmm-spaces, $\FMI$]
	We call\/ $\smallx=(X,r,\nu,\kappa)$ an \define{($I$-)functionally-marked metric measure space}
	(\define{fmm-space}) if $(X,r)$ is a complete, separable metric space, $\nu \in \Mf(X)$,
	and $\kappa \colon X \to I$ is measurable.
	We identify $\smallx$ with the marked metric measure space $(X, r, \mu)\in \MMI$, where $\mu$ satisfies
	\eqref{mark-fcn}. With a slight abuse of notation, we write $(X,r,\nu,\kappa)=(X,r,\mu)$ if
	\eqref{mark-fcn} is satisfied.
	Denote by $\FMI\subseteq \MMI$ the space of (equivalence classes of) fmm-spaces.
\end{Definition}

A first, simple observation is that $\FMI$ is a dense subspace of $\MMI$.

\begin{lemma}\label{lem:dense}
	The subspace\/ $\FMI$ is dense in\/ $\MMI$ with marked Gromov-weak topology.
\end{lemma}
\begin{proof}
	For $\smallx=(X,r,\mu) \in \MMI$, define $\smallx_n=(X\times I, r_n, \nu_n, \kappa_n)\in\FMI$ with $\nu_n=\mu$,
	$\kappa_n(x,u) = u$, and $r_n\((x,u), (y,v)\) := r(x,y) + e^{-n}\land d(u,v)$, for $x,y\in X,\;u,v\in I$.
	It is easy to see that $\smallx_n \rightarrow \smallx$ in the marked Gromov-weak topology.
\end{proof}

\subsection{The equicontinuous case}\label{sub:equicont} 

It directly follows from \lemref{dense} that the subspace $\FMI$ is not closed in $\MMI$, meaning that if
$\smallx_n \tomGw \smallx$ is a marked Gromov-weakly converging sequence in $\MMI$, and all $\smallx_n$ admit a
mark function, this need not be the case for $\smallx$.
In applications, however, the limit $\smallx$ is often not known explicitly, and it would be important to have
(sufficient) criteria for the existence of a mark function in terms of the $\smallx_n$ alone.
An easy possibility is Lipschitz equicontinuity: if all $\smallx_n$ admit a mark function that is Lipschitz
continuous with a common Lipschitz constant $L>0$, the same is true for $\smallx$ (see \cite{Piotrowiak:phd}).
More generally, this holds for uniformly equicontinuous mark functions as introduced below. We briefly discuss the equicontinuous
case in this subsection, because it is straightforward and illustrates the main ideas. 

Recall that a \emph{modulus of continuity} is a function $h\colon \R_+ \to \R_+ \cup\{\infty\}$ that is
continuous in $0$ and satisfies $h(0)=0$. A function $f\colon X \to I$, where $(X,r)$ is a metric space,
is \emph{$h$-uniformly continuous} if $d\(f(x),f(y)\) \le h\(r(x,y)\)$ for all $x,y \in X$.
Note that for every modulus of continuity $h$, there exists another modulus of continuity $h'\ge h$ which is
increasing and continuous with respect to the topology of the one-point compactification of
$\R_+$. Therefore, we can restrict ourselves without loss of generality to moduli of continuity from
\begin{equation}
\lbeq{eq:def-H}
	\CH:=\bsetbar{h\colon \R_+ \to \R_+\cup\{\infty\}}{h(0)=0,\; h\text{ is continuous and increasing}}.
\end{equation}
For $h\in\CH$ and a metric space $(X,r)$, we define
\begin{equation}\label{eq:goodset}
	A_h^X := A_h^{(X,r)} := \bset{(x_i,u_i)_{i=1,2}\in (X\times I)^2}{d(u_1,u_2) \le h(r(x_1,x_2))}
	\subseteq (X\times I)^2.
\end{equation}
Note that $f\colon X \to I$ is $h$-uniformly continuous if and only if $\((x, f(x)), (y, f(y))\) \in A_h^X$
for all $x, y \in X$, and that $A_h^X$ is a closed set in $(X\times I)^2$ with product topology.

\begin{definition}[$\HMI$]\label{d:HMI}
	For\/ $h\in\CH$, let\/ $\HMI\subseteq \FMI$ be the space of marked metric measure spaces admitting an\/
	$h$-uniformly continuous mark function.
\end{definition}

The next lemma states that a marked metric measure space $(X,r,\mu)$ admits an $h$-uniformly continuous mark
function if and only if a pair of independent samples from $\mu$ is almost surely in $A_h^X$. Furthermore, if a
sequence with $h$-uniformly continuous mark functions converges marked Gromov-weakly, the limit space also
admits an $h$-uniformly continuous mark function.

\begin{lemma}[uniform equicontinuity]\label{lem:equicont}
Fix a modulus of continuity\/ $h\in\CH$.
\begin{enumerate}
	\item\label{it:hmichar} $\displaystyle \HMI = \bset{(X,r,\mu)\in\MMI}{\musq(A_h^X) = \|\musq\|}.$
	\item\label{it:equicontcl} $\HMI$ is closed in the marked Gromov-weak topology. 
\end{enumerate}
\end{lemma}
\begin{proof}
	The mmm-space $\smallx=(X, r, \mu)$ is in $\HMI$ if and only if $\supp(\mu)$ is the graph of an $h$-uniformly
	continuous function. This is clearly equivalent to $\musq\((X\times I)^2 \setminus A_h^X\) = 0$.
	Item \itref{equicontcl} is obvious from \itref{hmichar}, because $A_h^X$ is a closed set.
\end{proof}

This preliminary result is quite restrictive because of the condition to have the same modulus of continuity
for all occurring spaces. In fact, the mark function of the tree-valued Fleming Viot
dynamic considered in \subref{FV} is not even continuous.

At the heart of the following generalisation to measurable mark functions lies the fact that measurable functions
are ``almost continuous'' by Lusin's celebrated theorem (see for instance \cite[Theorem~7.1.13]{BogachevII}).
Here, we give a version tailored to our setup:

\begin{LusinsThm}
	Let\/ $X,Y$ be Polish spaces, $\mu$ a finite measure on\/ $X$, and\/ $f\colon X \rightarrow Y$ a measurable
	function. Then, for every $\epsilon > 0$, there exists a compact set\/ $K_\epsilon \subseteq X$ such that\/
	$\mu(X \setminus K_\epsilon) < \epsilon$ and\/ $f\restricted{K_\epsilon}$ is continuous.
\end{LusinsThm}

\section{The space of fmm-spaces is Polish}\label{sec:Polish}

The subspace $\FMI$ is not closed in $\MMI$ in the marked Gromov-weak topology, and hence the restriction of the
marked Gromov-Prohorov metric $\dmGP$ to $\FMI$ is not complete. In this section, we show that there exists
another metric on $\FMI$ that induces the marked Gromov-weak topology and is complete. This shows that $\FMI$
is a Polish space in its own right.

\subsection{A complete metric on the space of fmm-spaces}

For a measure $\xi$ on $I$, we define 
\begin{equation}
	\beta_\xi := \probinta{I}{\probinta{I}{\left( 1 \land d(u,v) \right)}{\xi}{u}}{\xi}{v}.
\end{equation}
Note that $\beta_\xi=0$ if and only if $\xi$ is a Dirac measure.
For $\smallx=(X,r,\mu) \in \MMI$, with $\mu=\nu\otimes K$ as in \eqref{nuK}, we define
\begin{equation}\label{eq:beta}
	\beta(\smallx) := \probintanu{X}{\beta_{K_x}}{x}
		= \probintamu{X\times I}{\probinta{I}{\left( 1\land d(u,v) \right)}{K_x}{v}}{(x,u)}.
\end{equation}

\begin{prop}[characterization of $\FMI$ as continuity points]\label{p:usc}
	Let\/ $\cont(\beta)\subseteq \MMI$ be the set of continuity points of\/ $\beta\colon \MMI \to \R_+$,
	where\/ $\MMI$ carries the marked Gromov-weak topology. Then
	\begin{equation}\label{eq:cont}
		\cont(\beta) = \beta^{-1}(0) = \FMI.
	\end{equation}
\end{prop}
\begin{proof}[Proof (first part).]
	As seen before, $\smallx=(X,r,\nu\otimes K)\in\MMI$ admits a mark function if and only if $K_x$ is a
	Dirac measure for $\nu$-almost every $x\in X$, which is the case if and only if $\beta(\smallx)=0$.
	Hence $\beta^{-1}(0) = \FMI$.
	Because $\FMI$ is dense in $\MMI$ by \lemref{dense}, no $\smallx\in \MMI\setminus \beta^{-1}(0)$ can be
	a continuity point of $\beta$. Thus $\cont(\beta) \subseteq \beta^{-1}(0)$.
	
	We defer the proof of the inclusion $\beta^{-1}(0) \subseteq \cont(\beta)$ to \subref{betaestim},
	because it requires a technical estimate on $\beta$ derived in \propref{betaestim}.
\end{proof}

In view of \eqref{eq:cont}, we can use standard arguments to construct a complete metric on $\FMI$ that metrizes
marked Gromov-weak topology. Namely consider the sets 
\begin{equation}\label{Bdelta}
	F_m := \overline{\beta^{-1}\([\tfrac1m, \infty)\)} \subseteq \MMI, \quad m\in\N,
\end{equation}
where the closure is in the marked Gromov-weak topology.
Then, due to \propref{usc}, $F_m$ is disjoint from $\FMI$, and $\FMI = \MMI \setminus
\bigcup_{m\in\N} F_m$. Because $F_m$ is also closed by definition, we obtain
\begin{equation}\label{eq:rhogt0}
	\FMI = \bigcap_{m\in\N}\bset{\smallx \in \MMI}{\dmGP(\smallx, F_m) > 0}.
\end{equation}
We consider the metric $\dfmi$ on $\FMI$ defined for $\smallx, \smally \in \FMI$ by
\begin{equation}\label{dfmi}
	\dfmi(\smallx, \smally) := \dmGP(\smallx, \smally) + \sup_{m\in\N} 2^{-m} \land
	\Betrag{\frac1{\dmGP(\smallx, F_m)} - \frac1{\dmGP(\smally, F_m)}}.
\end{equation}

\begin{theorem}[$\FMI$ is Polish]\label{t:Polish}
	The space\/ $\FMI$ of\/ $I$-functionally-marked metric measure spaces with marked Gromov-weak topology is
	a Polish space. Namely, $\dfmi$ is a complete metric on\/ $\FMI$ inducing the marked
	Gromov-weak topology. 
\end{theorem}
\begin{proof}
	First, we show that $\dfmi$ induces the marked Gromov-weak topology on $\FMI$. For $m\in\N$, $\smallx\in
	\MMI$, define
	\begin{equation}\label{eq:rhodef}
		\rho_m(\smallx) :=\dmGP(\smallx, F_m),
	\end{equation}
	with $F_m$ defined in \eqref{Bdelta}.
	Note that $\rho_m$ is a continuous function on $\MMI$. 
	Let $\smallx_n, \smallx \in \FMI$. Then $\rho_m(\smallx) >0$ for all $m\in\N$ because of \eqref{eq:rhogt0}.
	Therefore, by definition, $\dfmi(\smallx_n, \smallx) \ton 0$ if and only if the two conditions
	$\dmGP(\smallx_n, \smallx) \ton 0$ and
	\begin{equation}\label{rhoconv}
		 \rho_m(\smallx_n) \ton \rho_m(\smallx) \qquad \forall m\in\N
	\end{equation}
	hold.  We have to show that the marked Gromov-weak convergence already implies \eqref{rhoconv}. This,
	however, follows from the continuity of the $\rho_m$.

	It remains to show that $\dfmi$ is a complete metric on $\FMI$.
	Consider a $\dfmi$-Cauchy sequence $(\smallx_n)_{n\in\N}$ in $\FMI$. By completeness of $\dmGP$ on
	$\MMI$, it converges marked Gromov-weakly to some $\smallx=(X, r, \mu)\in \MMI$.
	Furthermore, for every fixed $m\in\N$, \eqref{dfmi} implies that $1/\rho_m(\smallx_n)$ converges as
	$n\to \infty$, and hence $\dmGP(\smallx_n, F_m)$ is bounded away from zero. Thus $\smallx \not\in F_m$.
	Because $\FMI= \MMI\setminus \bigcup_{m\in\N} F_m$, this means that $\smallx\in\FMI$, and by the first
	part of the proof $\dfmi(\smallx_n, \smallx) \ton 0$.
\end{proof}

With $\BMMI(\smallx) := \bset{\smally\in \MMI}{\dmGP(\smallx, \smally) < \delta}$ we denote the open
$\delta$\nobreakdash-ball in\/
$\MMI$ with respect to $\dmGP$. The following corollary gives formal criteria for a limiting space to admit a
mark function, which are useful only together with estimates on $\beta$.

\begin{corollary}\label{c:complete}
	Let\/ $\folge\smallx$ be a sequence in\/ $\MMI$ which converges marked Gromov-weakly to\/ $\smallx$.
	Then the following four conditions are equivalent:
	\begin{enumerate}
		\item\label{it:fct} $\smallx\in \FMI$.
		\item\label{it:rho} $\nlimsup\rho_m(\smallx_n) > 0$ for all\/ $m\in\N$, with\/ $\rho_m$ defined
			in \eqref{eq:rhodef}.
		\item\label{it:betainv} For every $\delta > 0$, 
			\begin{equation}
				\nlimsup \inf_{\smally \in \beta^{-1}([\delta, \infty[)}\dmGP(\smallx_n, \smally) > 0.
			\end{equation}
		\item\label{it:betareg}
			\begin{equation}\label{betaneighbour}
				\lim_{\delta\downarrow0}\, \nliminf \sup_{\smally \in \BMMI(\smallx_n)} \beta(\smally) = 0.
			\end{equation}
	\end{enumerate}
\end{corollary}
\begin{proof}
	\eproofcase{\ref{it:fct}}{\ref{it:rho}}
		We have $\rho_m(\smallx) = \nlim \rho_m(\smallx_n)$, and
		$\rho_m(\smallx) > 0$ for all $m\in\N$ if and only if $\smallx \in \FMI$.
	\eproofcase{\ref{it:rho}}{\ref{it:betainv}} follows directly from the definition of $\rho_m$.
	\eproofcase{\ref{it:betainv}}{\ref{it:betareg}} Using monotonicity in $\delta$ we obtain
	\begin{align}
		\text{\ref{it:betainv}} &\iff \forall \delta>0\,\exists \eps>0\,\forall \folge\smally \subseteq \MMI
		    \text{ with\/ } \beta(\smally_n) \ge \delta :  \nlimsup \dmGP(\smallx_n, \smally_n) \ge \eps \\
		& \iff \forall \delta>0\, \exists\eps>0\,\forall \folge\smally\subseteq \MMI:
		    \nliminf \beta(\smally_n) < \delta \text{ or\/ } \nlimsup \dmGP(\smallx_n, \smally_n)\ge\eps
		    \nn\\
		& \iff \forall \eps>0\,\exists\delta>0\, \forall \folge\smally \subseteq \MMI \text{ with\/ }
		\smally_n\in\BMMI(\smallx_n) :
		    \nliminf \beta(\smally_n) < \eps \;\iff\; \text{\ref{it:betareg}}, \nn
	\end{align}
	where, in the third equivalence, we renamed $\delta$ to $\eps$ and $\eps$ to $\delta$.
\end{proof}

\subsection{A decomposition of $\FMI$ into closed sets and estimates on $\beta$}\label{sub:betaestim}

In this subsection, we derive some estimates on $\beta$ and use them to complete the proof of \propref{usc}.
Furthermore, we construct a decomposition of $\FMI$ into closed sets which are related to the sets $\HMI$.

As we have seen in \subref{equicont}, the situation becomes easy if we restrict to the uniformly equicontinuous
case, that is to the subspace $\HMI$ for some $h\in\CH$ as in \defref{HMI}. We introduce in what follows several related
subspaces capturing some aspect of equicontinuity. In analogy to the definition of $A_h^X$ in
\eqref{eq:goodset}, we use for a metric space $(X,r)$, and $\delta,\eps>0$, the notation
\begin{equation}\label{eq:epsgoodset}
	\Ade := \Ade[(X,r)] := \bset{(x_i,u_i)_{i=1,2}\in (X\times I)^2}{r(x_1,x_2)\ge
	\delta \text{ or } d(u_1,u_2) \le \eps} \subseteq (X\times I)^2.
\end{equation}
Note that $\Ade$ is a closed set. For every $h\in\CH$, using monotonicity and continuity of $h$, we observe that
\begin{equation}
	A_h^X = \bigcap_{\delta>0} \Adh.
\end{equation}

\begin{definition}[$\DEMI$, $\Mde$, $\Mh$]\label{d:EpsMI}
Let\/ $\delta, \eps > 0$ and $h\in\CH$. We define
\begin{align}
	\DEMI &:= \bset{(X,r,\mu)\in\MMI}{\musq(\Ade) = \|\mu^{\otimes 2}\|}, \\
	\Mde &:= \bset{(X,r,\mu)\in\MMI}
		{\exists \mu' \in \Mf(X \times I):\mu'\le \mu,\, \|\mu-\mu'\|\le\eps,\, (X,r,\mu') \in \DEMI},
\end{align}
and\/ $\Mh := \bigcap_{\delta>0} \Mdh$.
\end{definition}

The intuition is that for spaces in $\DEMIfull{\delta}{h(\delta)}$, the measure behaves as if it admitted 
an $h$-uniformly continuous mark function when distances of order $\delta$ are observed. The same holds for the
spaces in $\Mdh$ if we are additionally allowed to neglect a portion $h(\delta)$ of mass.

\begin{enremark}
	\item Clearly $\HMI \subseteq \Mh$. We will see in \lemref{cupMh} that $\Mh \subseteq \FMI$.
	\item The space $\Mh$ is \emph{much} larger than $\HMI$: while $\bigcup_{h\in\CH} \HMI$ contains only mmm-spaces
		admitting a uniformly \emph{continuous} mark function, we will see in \lemref{cupMh} that every element
		of $\FMI$ is in some $\Mh$.
	\item The spaces $\DEMI$ and $\Mde$ are not contained in $\FMI$. For instance, consider $I=\R$ and
	$\smallx=\(\{0\}, 0, \delta_{(0,0)} + \delta_{(0,\eps)}) \in \DEMI\subseteq \Mde$.
\end{enremark}

We have the following stability of $\Mde$ with respect to small perturbations in the marked Gromov-Prohorov
metric.
\begin{lemma}[perturbation of $\Mde$]\label{lem:perturb}
	Let\/ $\delta,\eps>0$, $\smallx\in \Mde$ and\/ $\smallxh\in \MMI$. Then
	\begin{equation}\label{eq:perturb}
		\delta' := \dmGP(\smallx, \smallxh) < \tfrac12\delta
			\;\implies\; \smallxh\in \Mdefull[\delta-2\delta'\!]{\,\eps+2\delta'}.
	\end{equation}
\end{lemma}
\begin{proof}
	Let $\smallx=(X,r,\mu)$, $\smallxh=(\Xh,\rh,\muh)$. We may assume that $X, \Xh$ are subspaces of some
	separable, metric space $(E,r_E)$ such that $\dPr(\mu,\muh) < \delta'$.
	By definition of $\Mde$, there is $\mu'\le \mu$ with
	$\|\mu-\mu'\|\le \eps$ and $\smallx':=(X,r,\mu') \in \DEMI$. Due to \lemref{rectangle}, we find $\muh'
	\le \muh$ with $\|\muh-\muh'\|\le \eps$ and $\dPr(\mu',\muh') < \delta'$, where $\smallxh{}'=(\Xh,\rh,\muh')$.
	By the coupling representation of the Prohorov metric, \eqref{eq:Prcouple},
	we obtain a measure $\xi$ on $(E\times I)^2$ with marginals $\xi_1\le \mu'$ and $\xi_2\le \muh'$ such that
	$\|\muh' - \xi_2\|\le \delta'$ and
	\begin{equation}
	\lbeq{xi-zero}
		\xi\(\bset{\((x,u), (\xh, \uh)\) \in (X \times I) \times (\Xh \times I)}{ r_E(x,\xh)+d(u,\uh) \ge \delta'}\) = 0.
	\end{equation}
	By definition, $\mupsq$ is supported by $A_{\delta,\eps}^X$. Therefore, the same is true for
	$\xi_1^{\otimes 2}$ and
	we obtain
  \begin{align}
  	\|\xi_2^{\otimes 2} \| &= \|\xi^{\otimes 2} \|
		= \xi^{\otimes 2}\(\bset{(x_i, u_i,\xh_i,\uh_i)_{i=1,2} \in ((X \times I) \times (\Xh \times I))^2}{(x_i,u_i)_{i=1,2} \in \Ade}\) \\
		&\le \xi^{\otimes 2}_2\(\bset{(\xh_i,\uh_i)_{i=1,2} \in (\Xh\times I)^2}{ r_E(\xh_1,\xh_2) \geq \delta-2\delta'
			\text{ or } d(\uh_1,\uh_2) \leq \eps+2\delta'}\) \nn\\
		&= \xi_2^{\otimes 2}(A_{\delta-2\delta',\eps+2\delta'}^{\Xh}), \nn
  \end{align}
  where the inequality follows from \eqref{xi-zero} together with the triangle-inequality.
  Therefore, 
	$(\Xh,\rh,\xi_2) \in \DEMIfull{\delta-2\delta'\!}{\,\eps+2\delta'}$.
  Now the claim follows from $\|\muh - \xi_2\| \le \|\muh-\muh'\| + \|\muh'-\xi_2\| \le \eps + \delta'$.
\end{proof}

\begin{prop}[estimates on $\beta$]\label{p:betaestim}
	Let\/ $\delta, \eps >0$ and consider $\smallx=(X,r,\mu)\in \MMI$.
	Then the following hold:
\begin{enumerate}
	\item\label{it:varLip} If\/ $\mu'\in\Mf(X \times I)$, then\/ $\beta(\smallx) \le \beta\((X,r,\mu')\) + 2\|\mu-\mu'\|$.
	\item\label{it:DEMIest} If\/ $\smallx\in \DEMI[2\delta]$, then\/ $\beta(\smallx) \le \eps\|\mu\|$.
	\item\label{it:estimate-beta-3} If\/ $\smallx\in \Mde[2\delta]$ and\/ $\smallxh\in\MMI$ with\/
		$\dmGP(\smallx, \smallxh)<\delta$, then\/ $\beta(\smallx) \le \eps\(\|\mu\|+2\)$ and 
	\begin{equation}\label{eq:ballest}
		\beta(\smallxh) \le (\eps+2\delta)(2+\|\mu\| + \delta).
	\end{equation}
\end{enumerate}
\end{prop}
\begin{enproof}
\item follows directly from the definition.
\item If $x \in X$ and $u,v\in I$ satisfy $\((x,u),\,(x,v)\) \in \Adtwoe$, then $d(u,v) \le \eps$ by definition of $\Adtwoe$.
	Thus $\beta(\smallx) = \probintamu{X\times I}{\probinta{I}{(1\land d(u,v))}{K_x}{v}}{(x,u)} \le \eps\|\mu\|$.
\item Combining \itref{varLip} and \itref{DEMIest} yields $\beta(\smallx) \le 2\eps + \eps\|\mu\|$.
	Let $\delta'=\dmGP(\smallx,\smallxh)$.  By \lemref{perturb}, we have
	$\smallxh\in \Mdefull[2\delta-2\delta'\!]{\,\eps+2\delta'}$ and thus $\beta(\smallxh) \le
	(2+\|\muh\|)(\eps+2\delta')\le(2+\|\mu\|+\delta)(\eps+2\delta)$.
\end{enproof}

In order to complete the proof of \propref{usc} with the help of \propref{betaestim}, we first observe that, as
a consequence of Lusin's theorem, every functionally marked metric measure space is an element of $\Mh$ for some
$h\in\CH$. Together with \lemref{Mhclosed} below, this means that we have a nice (though uncountable)
decomposition of $\FMI$ into closed sets.

\begin{lemma}[decomposition of $\FMI$]\label{lem:cupMh}
	The following equality holds: $\FMI=\bigcup_{h\in\CH} \Mh$.
\end{lemma}
\begin{proof}
	We have $\Mh\subseteq\beta^{-1}(0)=\FMI$ for every $h\in\CH$. Indeed, the equality was
	shown in the first part of the proof of \propref{usc}. To obtain the inclusion, that is
	$\beta(\smallx)=0$ for all $\smallx \in \Mh$, recall $\Mh$ from \defref{EpsMI} and choose
	$\eps=h(2\delta)$ in \propref{betaestim}\itref{estimate-beta-3}.
	
	Conversely, let $\smallx=(X,r,\nu,\kappa)\in \FMI$. According to Lusin's theorem, we find for every
	$\eps>0$ a compact set $K_\eps\subseteq X$, and a modulus of continuity $h_\eps\in\CH$, such that
	$\nu(X\setminus K_\eps) \le \eps$ and $\kappa\restricted{K_\eps}$ is $h_{\eps}$-uniformly continuous.
	In particular,
	\begin{equation}\label{eq:xinm}
		\smallx\in \Mdefull{h_\eps(\delta)\lor\eps} \quad\forall \eps,\delta>0.
	\end{equation}
	We may assume without loss of generality that $\eps\mapsto h_\eps(\delta)$ is decreasing and
	right-continuous for every $\delta>0$. We define
	\begin{equation}
		h(\delta) := \inf\bset{\eps>0}{h_\eps(\delta) < \eps} \in \R_+\cup\{\infty\}.
	\end{equation}
	Clearly, $h(\delta)$ converges to $0$ as $\delta\downarrow0$ because $h_\eps\in\CH$.
	Furthermore, $h_{h(\delta)}(\delta) \le h(\delta)$, and hence \eqref{eq:xinm} with $\eps=h(\delta)$
	implies $\smallx \in \Mh$.
\end{proof}

\begin{proof}[Proof of \propref{usc} (completion).]
	We still have to show continuity of $\beta$ in $\smallx \in \beta^{-1}(0)$.
	Due to \lemref{cupMh}, there is $h\in\CH$ with $\smallx\in\Mh$. Now \propref{betaestim} yields for
	$\delta>0$ the estimate
	$\sup_{\smallxh\in\BMMI(\smallx)}\beta(\smallxh) \le (h(2\delta) + 2\delta)(2+\|\mu\| + \delta)$,
	which converges to $0$ as $\delta \downarrow 0$.
\end{proof}

It directly follows from \propref{betaestim}\itref{estimate-beta-3} that the marked Gromov-weak closure of $\Mh$ is
contained in $\FMI$. In fact, $\Mh$ is even Gromov-weakly closed, which will be used in the
proof of \thmref{modcadlag} below.

\begin{lemma}[closedness of $\Mh$]\label{lem:Mhclosed}
	For every\/ $\delta,\eps>0$, $\Mde$ is marked Gromov-weakly closed in\/ $\MMI$. In particular, $\Mh$ is closed
	for every\/ $h\in\CH$.
\end{lemma}
\begin{proof}
	Fix $\eps,\delta>0$ and let $\folge\smallx$ be a sequence in $\Mde$ converging marked Gromov-weakly to
	some $\smallx=(X,r,\mu)\in\MMI$. Using \lemref{seqembed}, we may assume that $X_n$, $n\in\N$, and $X$ are
	subspaces of a common separable, metric space $(E,r_E)$, such that $\mu_n \tow \mu$ on $E\times I$. By
	definition of $\Mde$, we find $\mu_n' \le \mu_n$, $\|\mu_n' - \mu_n\| \le \eps$, such that $\munpsq$ is
	supported by $\Ade[E]$ for all $n\in\N$. Since $(\mu_n')_{n\in\N}$ is tight, we may assume, by passing
	to a subsequence, that $\mu_n' \tow \mu'$ for some $\mu'\in \Mf(E)$. Obviously, $\mu'\le \mu$ and
	$\|\mu-\mu'\| = \nlim \|\mu_n\| - \|\mu'_n\|  \le \eps$. Because $\Ade[E]$ is closed, $\mupsq$ is
	supported by $\Ade[E]$ and hence $\smallx \in \Mde$.
\end{proof}

\section{Criteria for the existence of mark functions}\label{sec:criteria}

Based on the construction of the complete metric and the decomposition $\FMI=\bigcup_{h\in\CH}\Mh$ into closed
sets obtained in \secref{Polish}, we now derive criteria to check if a marked metric measure space admits a mark
function, especially in the case where it is given as a marked Gromov-weak limit.
We then transfer the results to random mmm-spaces and $\MMI$-valued stochastic processes.

\subsection{Deterministic criteria} \label{sub:criteria-1}

Our main criterion for deterministic spaces is a direct consequence of the results in \secref{Polish}.
Recall that $\CH$ is the set of moduli of continuity defined in \eqref{eq:def-H}.

\begin{theorem}[characterization of existence of a mark function in the limit]\label{t:modulus}
	Let\/ $\folge\smallx$ be a sequence in\/ $\MMI$ with\/ $\smallx_n \tomGw \smallx\in\MMI$.
	Then $\smallx \in \FMI$ if and only if there exists\/ $h\in\CH$ such that for every\/ $\delta>0$ 
	\begin{equation}\label{eq:mod}
		\smallx_n \in \Mdh \quad\text{ for infinitely many\/ $n\in\N$.}
	\end{equation}
	In this case, $\smallx\in \Mh$. 
\end{theorem}
\begin{proof}
	First assume there is $h\in\CH$ such that \eqref{eq:mod} is satisfied.
	Since $\Mdh$ is closed by \lemref{Mhclosed}, \eqref{eq:mod} implies that $\smallx \in \Mdh$ for every
	$\delta$, that is $\smallx \in \Mh$. By \lemref{cupMh}, $\Mh\subseteq \FMI$.

	Conversely, assume $\smallx\in\FMI$. Then, by \lemref{cupMh}, we find $h\in \CH$ with $\smallx \in \Mh$.
	We claim that \eqref{eq:mod} holds with $h$ replaced by $\hh(\delta) := h(3\delta) + 2\delta$.
	Indeed, fix $\delta>0$ and observe that $\smallx\in \Mh\subseteq\Mdh[3\delta]$.
	\lemref{perturb} yields $\smallx_n \in\Mdheps[\hh]$ for all $n$ with $\dmGP(\smallx, \smallx_n) < \delta$.
\end{proof}

We will use \thmref{modulus} in the following form.

\begin{corollary}\label{c:modulus}
	Let\/ $\smallx_n=(X_n, r_n, \nu_n, \kappa_n)\in\FMI$, $\smallx_n \tomGw \smallx\in \MMI$.
	Let\/ $Y_{n,\delta} \subseteq X_n$ measurable for\/ $n\in\N,\, \delta>0$, and\/ $h\in\CH$.
	Then $\smallx \in \FMI$ if the following two conditions hold for every\/ $\delta>0$:
	\begin{gather}
		\nliminf \nu_n(X_n\setminus Y_{n,\delta}) \le h(\delta), \label{eq:modulus1}\\
		\forall n\in\N,\, x,y\in Y_{n,\delta}: r_n(x,y) < \delta \implies
			d\(\kappa_n(x),\kappa_n(y)\) \le h(\delta). \label{eq:modulus2}
	\end{gather}
\end{corollary}
\begin{proof}
	Let $\mu_n' := \mu_n\restricted{Y_{n,\delta}\times I}$, where $\mu_n=\nu_n\otimes \delta_{\kappa_n}$.
	Then \eqref{eq:modulus2} implies $(X_n,r_n,\mu_n') \in\DEMIfull{\delta}{h(\delta)}$ and
	\eqref{eq:modulus1} yields $\|\mu_n' - \mu_n\| \le h(\delta)$ for infinitely many $n$.
	Hence we can apply \thmref{modulus}.
\end{proof}

\begin{remark}\label{r:seq}
	To obtain $\smallx \in\FMI$, it is clearly enough to show in \thmref{modulus} and \corref{modulus},
	\eqref{eq:mod} respectively \eqref{eq:modulus1} and \eqref{eq:modulus2} only for $\delta=\delta_m$ for
	a sequence $\folge[m]\delta$ with $\delta_m \downarrow 0$ as $m\to \infty$.
\end{remark}

We illustrate the r\^ole of the exceptional set $X_n \setminus Y_{n,\delta}$, and the importance of its dependence on $\delta$,
with a simple example.

\begin{example}
	Consider\/ $X=[0,1]$ with Euclidean metric $r$, $\nu=\lambda + \delta_0$, where $\lambda$ is
	Lebesgue-measure, and $\kappa_n(x) = (nx) \land 1$. Obviously, $\smallx_n=(X,r,\nu,\kappa_n)$
	converges marked Gromov-weakly and the limit admits the mark function $\1_{(0,1]}$. To see this from
	\corref{modulus}, we choose $h(\delta)=\delta$ and $Y_{n,\delta}=\{0\} \cup [\delta \vee \tfrac{1}{n} , 1]$. Note that we cannot choose
	$Y_{n,\delta}$ independent of $\delta$.
\end{example}

\begin{remark}[equicontinuous case]
	 If, in \corref{modulus}, $Y_{n,\delta}=Y_n$ does not depend on $\delta$, then \eqref{eq:modulus2}
	 means that $\kappa_n$ is $h$-uniformly continuous on $Y_n$. Consequently, the mark function of
	 $\smallx$ is in this case $h$-uniformly continuous. If we restrict to $Y_n = X_n$ for all $n$, we
	 recover part \itref{equicontcl} of \lemref{equicont}.
\end{remark}

\begin{corollary}\label{c:diamcrit}
	Let\/ $\smallx_n=(X_n, r_n, \nu_n, \kappa_n)\in \FMI$ and assume that\/ $\smallx_n$ converges to\/
	$\smallx=(X,r,\mu)\in \MMI$ marked Gromov-weakly. Further assume that for\/ $n\in\N,\, \delta>0$, there
	are measurable sets\/ $Z_{n,\delta}\subseteq X_n$, such that
	\begin{equation}\label{eq:diamcrit}
		\lim_{\delta\downarrow0}\, \nliminf\, \biggl(
			\nu_n(X_n\setminus Z_{n,\delta}) +
		  \probinta{Z_{n,\delta}}{\left( 1\land\diam\(\kappa_n\(B_{\delta}^{X_n}(x)\cap Z_{n,\delta}\)\)\right) }{\nu_n}{x}
		  \biggr) \,=\, 0,
	\end{equation}
	where $\diam$ is the diameter of a set.
	Then $\smallx$ admits a mark function, that is $\smallx \in \FMI$.
\end{corollary}
\begin{proof}
	For $\delta>0$ let
	\begin{equation}
	\label{eq:diamcrit1}
	  g(\delta) := \sup_{0<\delta'\leq\delta} \nliminf\, \biggl(
	  	\nu_n(X_n\setminus Z_{n,\delta'}) +
		  \probinta{Z_{n,\delta'}}{\left( 1\land\diam\(\kappa_n\(B_{\delta'}^{X_n}(x)\cap Z_{n,\delta'}\)\) \right)}{\nu_n}{x}
		  \biggr).
	\end{equation}
	By \eqref{eq:diamcrit}, $\lim_{\delta\downarrow0} g(\delta)=0$ and $g$ is increasing with
	$\|g\|_\infty \leq \|\mu\|$. Let $h\in\CH$ be such that $g(\delta) \leq \frac{h(\delta)}{2} \bigl( 1\land
	h(\delta))$ for all $\delta>0$. Then
	\begin{equation}\label{eq:diamcrit2}
		\nu_n\(\bset{x\in Z_{n,\delta}}{\diam\(\kappa_n(B_{\delta}^{X_n}(x)\cap Z_{n,\delta})\) > h(\delta)}\)
		\leq \frac{g(\delta)}{1\land h(\delta)} \leq h(\delta)/2.
	\end{equation}
	Now apply \corref{modulus} with
	\begin{equation} \label{eq:diamcrit3}
		Y_{n,\delta} := \bset{x \in Z_{n,\delta}}{\diam\(\kappa_n\(B_{\delta}^{X_n}(x)\cap
			Z_{n,\delta}\)\) \leq h(\delta)}. 
	\end{equation}
	Then \eqref{eq:modulus2} follows from the definition of $Y_{n,\delta}$ in \eqref{eq:diamcrit3},
	and $\nu_n(X_n\setminus Y_{n,\delta}) \le \nu_n(X_n\setminus Z_{n,\delta}) + h(\delta)/2 \le g(\delta) +
	h(\delta)/2 \le h(\delta)$ holds by \eqref{eq:diamcrit2} and \eqref{eq:diamcrit1}. 
\end{proof}

\subsection{Random fmm-spaces} \label{sub:criteria-2}

The following theorem is a randomized version of \thmref{modulus}. It is our main criterion for $\MMI$-valued
random variables.

\begin{theorem}[random fmm-spaces as limits in distribution]\label{t:rnd-modulus}
	Let\/ $\folge\X$ be a sequence of\/ $\MMI$-valued random variables which converges in distribution (w.r.t.\
	marked Gromov-weak topology) to an\/ $\MMI$-valued random variable\/ $\X$.
	Further assume that for every\/ $\eps>0$, there exists a modulus of continuity\/ $h_\eps\in \CH$ such
	that
	\begin{equation}\label{eq:pasgen}
		\deltalimsup\, \nlimsup\, \bPs{\X_n \in \Mdheps} \ge 1-\eps.
	\end{equation}
	Then\/ $\X$ admits almost surely a mark function, that is\/ $\X\in\FMI$ almost surely.

	If additionally\/ $\X_n=(X_n,r_n,\nu_n,\kappa_n)\in\FMI$ almost surely for all\/ $n\in\N$, we can replace
	\eqref{eq:pasgen} by existence of random measurable sets\/
	$Y_{n,\delta}^\eps \subseteq X_n$, $n\in\N,\, \delta>0$, in addition to the $h_\eps\in\CH$, such that
	the following two conditions hold for every\/ $\eps>0$:
	\begin{gather}
		\deltalimsup\, \nlimsup\, \bPs{\nu_n(X_n\setminus Y_{n,\delta}^\eps) \le h_\eps(\delta)}
			\ge 1-\eps. \label{eq:pas}\\
		\forall n\in\N,\, \delta>0,\, x,y\in Y_{n,\delta}^\eps: r_n(x,y) < \delta \implies
			d\(\kappa_n(x),\kappa_n(y)\) \le h_\eps(\delta). \label{eq:moduluseps}
	\end{gather}
\end{theorem}

\begin{remark}
	In \eqref{eq:pas}, we need not worry about measurability of the ``event''
	$B_{n,\delta} := \bigl\{\nu_n(X_n\setminus Y_{n,\delta}^\eps) \le h_\eps(\delta)\bigr\}$  due to the choice of $Y_{n,\delta}^\eps$.
	The inequality \eqref{eq:pas} is to be understood in the sense of inner measure, that is we require that there are
	measurable sets $C_{n,\delta}\subseteq B_{n,\delta}$ with
	$\deltalimsup\nlimsup\P(C_{n,\delta}) \ge 1-\eps$.
\end{remark}

\begin{proof}
	The second statement follows in the same way as \corref{modulus}.
	We divide the proof of the main part in two steps. First, we show $\X\in\FMI$ if, instead of
	\eqref{eq:pasgen}, even
	\begin{equation}\label{eq:deltainside}
		\P\Bigl(\bigcap_{m\in\N} \bigl\{ \X_n \in \Mdmh \text{ for infinitely many $n$}\bigr\} \Bigr)
			\ge 1-\eps
	\end{equation}
	holds for a sequence $\delta_m=\delta_m(\eps)\downarrow 0$ as $m\to\infty$. In the second step, we show
	that, given \eqref{eq:pasgen}, we can modify $h_\eps$ to $\hh_\eps \in \CH$ such that
	\eqref{eq:deltainside} holds with $h_\eps$ replaced by $\hh_\eps$.
	\begin{proofsteps}
	\step By Skorohod's representation theorem, we may assume that the $\X_n$ are coupled such that
		they converge almost surely to $\X$ in the marked Gromov-weak topology.
		The inequality \eqref{eq:deltainside} implies that with probability at least $1-\eps$, for all
		$m \in \NN$, $\X_n \in \Mdmh$ for infinitely many $n$.
		By \thmref{modulus} and \remref{seq}, this means that the probability that $\X$ admits a mark
		function is at least $1-\eps$. Because $\eps$ is arbitrary, this implies $\X\in\FMI$ almost surely.
	\step Let $T(\eps,\delta):= \nlimsup\, \bPs{\X_n \in \Mdheps}$ in \eqref{eq:pasgen}. Set
	\begin{equation}
	  \delta_1 := \sup\bigl\{ \delta \in [0,1] : T(\eps/4,\delta) \ge 1-\eps/2 \mbox{ and } h_{\eps/4}(\delta)<1 \bigr\}. 
	\end{equation}
By \eqref{eq:pasgen} and as $h_{\eps/4} \in \CH$, the set inside the supremum is non-empty. Next define recursively
  \begin{equation}
    \delta_m := \sup\bigl\{ \delta \in [0,\delta_{m-1}/2] :
    	T(\eps 2^{-(m+1)},\delta) \ge 1-\eps 2^{-m} \mbox{ and } h_{\eps 2^{-(m+1)}}(\delta)<1/m \bigr\}
  \end{equation}
for $m \in \NN, m \geq 2$. Again, the set inside the supremum is non-empty by \eqref{eq:pasgen} and as $h_{\eps 2^{-(m+1)}} \in \CH$. Moreover, $\delta_m=\delta_m(\eps)>0$, $\delta_m \downarrow 0$ for $m \rightarrow \infty$ and $h_{\eps 2^{-(m+1)}}(\delta_m) \leq 1/m$ follows. We can therefore set
		\begin{equation}
			\hh_\eps(\delta_m) := h_{\eps2^{-(m+1)}}(\delta_m)
		\end{equation}
and extend this to $\hh_\eps\in\CH$.
		Using Fatou's lemma, we obtain
		\begin{align}
		\P\Bigl(\bigcup_{m\in\N} \bigl\{\X_n \not\in \Mdmh[\hh_\eps] \text{ eventually}\bigr\} \Bigr)
			& \le \sum_{m\in\N} \EE\(\nliminf \1_{\MMI\setminus\Mdmh[\hh_\eps]}(\X_n)\) \\
			& \le \sum_{m\in\N} \nliminf \bPs{\X_n \not\in \Mdmh[\hh_\eps]} \nn\\
			& = \sum_{m\in\N} \big( 1-T(\eps 2^{-(m+1)},\delta_m) \big) \nn\\
			& \le \sum_{m\in\N} \eps2^{-m} = \eps. \nn
		\end{align}
		Thus \eqref{eq:deltainside} holds with $h_\eps$ replaced by $\hh_\eps$.
	\end{proofsteps}
\end{proof}

\subsection{Fmm-space-valued processes} \label{sub:criteria-3}

Let $J\subseteq\R_+$ be a (closed, open or half-open) interval and consider a stochastic process
$\X=(\X_t)_{t\in J}$ with values in $\MMI$ and c\`adl\`ag paths, where $\MMI$ is equipped with the marked
Gromov-weak topology.
We say that $\X$ is an \emph{$\FMI$-valued c\`adl\`ag process} if
\begin{equation}\label{eq:fmi-val}
	\bPs{\X_t,\X_{t-} \in \FMI \text{ for all\/ } t\in J} = 1,
\end{equation}
where\/ $\X_{t-}$ is the left limit of\/ $\X$ at\/ $t$ ($\X_{\ell-} := \X_\ell$ if $\ell$ is the left endpoint of $J$).
In the following, we give sufficient criteria for $\X$ to be an $\FMI$-valued c\`adl\`ag process. We are
particularly interested in the situation where $\X$ is the limit of $\FMI$-valued processes $\X^n$.

Unsurprisingly, if the set of $\PP$-measure smaller or equal to $\eps$ in \thmref{rnd-modulus} 
is independent of $t$, the result is true for all $t$ simultaneously, almost surely. The modulus of
continuity may also depend on $t$ in a continuous way; or be arbitrary if the limiting process has continuous
paths:

\begin{theorem}\label{t:pr-modulus}
	Let\/ $J\subseteq\R_+$ be an interval, and\/ $\X^n=(\X^n_t)_{t\in J}$, $n\in\N$, a sequence of\/
	$\MMI$-valued \cadlag\ processes converging in distribution to an\/ $\MMI$-valued \cadlag\ process\/
	$\X=(\X_t)_{t\in J}$.  Assume that for every\/ $t\in J$, $\eps>0$, there exists\/ $h_{t,\eps}\in \CH$ such that
	\begin{equation}\label{eq:pasgenforall}
		\deltalimsup\, \nlimsup\, \bPs{\X^n_t \in \Mdheps[h_{t,\eps}]\;\, \forall t\in J} \ge 1-\eps.
	\end{equation}
	Then\/ $\X$ is an\/ $\FMI$-valued \cadlag\ process, that is \eqref{eq:fmi-val} is satisfied,
	if at least one of the following two conditions holds:
	\begin{enumerate}
		\item\label{it:cond1} $\X$ has continuous paths a.s.
		\item\label{it:cond2} $t\mapsto h_{t,\eps}(\delta)$ is continuous for every $\eps, \delta >0$.
	\end{enumerate}

	If additionally\/ $\X^n$ is\/ $\FMI$-valued almost surely for all\/ $n\in\N$, \eqref{eq:pasgenforall} can be
	replaced by existence of random measurable sets\/ $Y_{t,\eps,\delta}^n \subseteq X^n_t$, in addition to
	the\/ $h_{t,\eps}\in\CH$, satisfying the following two conditions for every\/ $\eps>0$:
	\begin{gather}
		\deltalimsup\, \nlimsup\,
		\bPs{\nu^n_t(X^n_t\setminus Y_{t,\eps,\delta}^n) \le h_{t,\eps}(\delta)\;\, \forall t\in J}
		\ge 1-\eps, \label{eq:pasforall} \\
		\forall n\in\N,\, t\in J,\, \delta>0,\, x,y\in Y_{t,\eps,\delta}^n: r_n(x,y) < \delta \implies
			d\(\kappa_n(x),\kappa_n(y)\) \le h_{t,\eps}(\delta). \label{eq:modulust}
	\end{gather}
\end{theorem}
\begin{proof}
Due to the Skorohod representation theorem, we may assume that $\X^n\to \X$ almost surely in the Skorohod
topology. For condition \itref{cond1} respectively \itref{cond2} we obtain
\begin{enumerate}
\item If $\X$ has continuous paths a.s., the convergence in Skorohod topology implies uniform convergence of
	$\X_t^n(\omega)$ on $J$ a.s. with respect to $\dmGP$.
	Hence we have $\X^n_t \tonmGw \X_t$ for all $t\in J$, almost surely, and we can proceed as in the proof
	of \thmref{rnd-modulus}.
\item There are (random) continuous $w^n \colon J \to J$, converging to the identity uniformly on
	compacta, such that $\X^n_{w^n(t)} \to \X_t$ for all $t\in J$, almost surely. We can use the moduli of
	continuity $\hh_{t,\eps}(\delta) := h_{t,\eps}(\delta) + \delta$ and proceed as in the proof of
	\thmref{rnd-modulus}. Note here that, due to continuity of $h_{t,\eps}(\delta)$ in $t$, there is for
	every compact subinterval $\mathcal{J}$ of $J$ an $N_{\mathcal{J},\eps,\delta}\in\N$ such that
	$\hh_{t,\eps}(\delta) \ge h_{w^n(t),\eps}(\delta)$ for all $n\ge N_{\mathcal{J},\eps,\delta}$ and
	$t\in \mathcal{J}$.

	The same arguments apply for left limits with $w^n_{-}$ such that $\X^n_{w^n_{-}(t)} \to \X_{t-}$.
\qedhere\end{enumerate}
\end{proof}

To use \thmref{pr-modulus}, we have to check in \eqref{eq:pasgenforall} or \eqref{eq:pasforall} a condition for
uncountably many $t$ simultaneously, which is often much more difficult than for every $t$ individually.
One situation, where it is easy to pass from individual $t$ to all $t$ simultaneously is the case where the
moduli of continuity $h_{t,\eps}$ actually do not depend on $t$ and $\eps$ (see \corref{epsindep}).
The independence of $\eps$, however, is a strong requirement.
Therefore, we relax it to not blowing up too fast as $\eps\downarrow 0$, where the ``too fast'' is determined by
the following modulus of c\`adl\`agness of the limiting process.

\begin{definition}[modulus of c\`adl\`agness]
Let\/ $J$ be an interval, $(E,r)$ a metric space, and\/ $e=(e_t)_{t\in J}\in\DE$ a \cadlag\ path on $J$ with values in $E$.
Following \textup{\cite[(14.44)]{Bil68}}, set
\begin{equation}
  w''(e,\delta) := \sup_{t,t_1,t_2 \in J: t_1 \leq t \leq t_2, t_2-t_1 \leq \delta}
  	\min\bigl\{ r(e(t),e(t_1)),\, r(e(t_2),e(t)) \bigr\}.
\end{equation}
We say that\/ $e$ \emph{admits $w\in\CH$ as modulus of c\`adl\`agness} if\/ $w''(e,\delta) \le w(\delta)$ for all\/
$\delta>0$.
\end{definition}

\begin{theorem}\label{t:modcadlag}
	Fix an interval\/ $J\subseteq \R_+$. Let\/ $\X=(\X_t)_{t\in J}$ and\/ $\X^n=(\X^n_t)_{t\in J}$,
	$n\in\N$, be\/ $\MMI$-valued c\`adl\`ag processes such that\/ $\X^n$ converges in distribution to\/ $\X$.
	Furthermore, assume that there is a dense set\/ $Q\subseteq J$ and\/ $w_\eps, h_\eps\in \CH$,
	such that for all\/ $\eps>0$
	\begin{gather}
		\nlimsup \Ps{\X_t^n \in \Mdheps} \ge 1-\eps \qquad \forall \delta>0,\,t\in Q, \label{eq:1}\\
		\Ps{t\mapsto \X_t \text{ admits\/ $w_\eps$ as modulus of c\`adl\`agness w.r.t.\ $\dmGP$}} \ge 1-\eps,\,
		\text{ and} \label{eq:2}\\
		\liminf_{\delta\downarrow 0} h_{\eps\cdot\delta}\(2w_\eps(\delta)\) = 0. \label{eq:3}
	\end{gather}
	Then\/ $\X$ is an\/ $\FMI$-valued c\`adl\`ag process, that is \eqref{eq:fmi-val} holds.
\end{theorem}

Recall the decomposition $\MMI\setminus \FMI = \bigcup_{m\in\N} F_m$ with $F_m$ defined in \eqref{Bdelta}.
The basic idea of the proof is to use the following lemma about \cadlag\ paths to show that, almost surely, the
path of\/ $\X$ avoids $F_m$. The assertion of the lemma follows easily using the triangle-inequality.

\begin{lemma}\label{lem:pathcontain}
	Let\/ $J$ be an interval, $(E,r)$ a metric space, and\/ $e=(e_t)_{t\in J}\in\DE$ a \cadlag\ path
	admitting \modcadlag\ $w\in\CH$. Let\/ $F\subseteq E$ be any set, $\delta>0$, and\/ $Q\subseteq J$ such
	that for all $t\in J$ there is $t_1,t_2\in Q$ with $t_1\le t \le t_2 \le t_1+\delta$. Then 
	\begin{equation}
		r(e_t, F) > w(\delta)\;\;\forall t \in Q \quad\implies\quad e_t \not\in F \text{ and\/ } e_{t-} \not\in F \;\;\forall t\in J.
	\end{equation}
\end{lemma}

\begin{proof}[Proof of \thmref{modcadlag}]
	Because $\Mheps=\bigcap_{\delta>0}\Mdheps$ is closed by \lemref{Mhclosed}, the Portmanteau theorem and \eqref{eq:1} imply 
	\begin{equation}\label{eq:lim1}
		\Ps{\X_t \not\in \Mheps} < \eps \qquad \forall t\in Q,\,\eps>0.
	\end{equation}
	Due to the Skorohod representation theorem, we may assume that $\X^n\to \X$ almost surely in Skorohod
	topology. In order to simplify notation, we assume $J=[0,1]$ and $Q=\bigcup_{k\in\N} Q_k$
	with $Q_k=\set{i2^{-k}}{i=0,\ldots,2^k}$. It is enough to show for every $\eps>0,\, m\in\N$ and $F_m$ as defined in \eqref{Bdelta} that
	\begin{equation}\label{eq:goal}
		p_m := \bPs{\exists t\in[0,1]: \X_t \mbox{ or } \X_{t-} \in F_m} \le 3\eps.
	\end{equation}

	To show \eqref{eq:goal}, fix $\eps>0$ and $m\in\N$, and let $\X_t=(X_t,r_t,\mu_t)$.
	Because $\X$ has \cadlag\ paths, we find $K=K(\eps)<\infty$ such that
	\begin{equation}\label{eq:K}
		\bPs{\sup_{t\in[0,1]} \|\mu_t\| \ge K-3} < \eps.
	\end{equation}
	According to \eqref{eq:3} and \eqref{eq:lim1}, we can choose $k\in \N$ big enough such that for $h:=h_{\eps2^{-k}}$
	we have
	\begin{equation}\label{eq:hestim}
		h\(2w_\eps(2^{-k})\) < (Km)^{-1} - 2 w_\eps(2^{-k}) \qquad\text{and}\qquad 
		\Ps{\X_t \not\in \Mh} < \eps 2^{-k}.
	\end{equation}
	Assume without loss of generality that $w_\eps(2^{-k}) \leq 1$.
	Now \propref{betaestim}\itref{estimate-beta-3} implies that, whenever $\X_t\in\Mh$ and $\|\mu_t\| < K-3$, we have
	\begin{equation}\label{eq:dmGPestim}
		\dmGP(\X_t, F_m) > w_\eps(2^{-k}).
	\end{equation}
	Combining \eqref{eq:2} and \lemref{pathcontain}, we obtain
	\begin{align}
		p_m &\le \eps + \bPs{\exists t\in Q_k: \dmGP(\X_t, F_m) \leq w_\eps(2^{-k})}. \\
	\intertext{Using \eqref{eq:K}, \eqref{eq:dmGPestim}, and (in the last step) \eqref{eq:hestim}, we conclude}
		p_m &\le 2\eps + 2^k\sup_{t\in Q_k} \bPs{\|\mu_t\|<K-3,\, \X_t\not\in \Mh} \le 3\eps.
	\end{align}
	Thus \eqref{eq:goal} holds for all $\eps>0$, and $\Ps{\exists t\in [0,1]:X_t\not\in \FMI}  = \sup_{m\in\N} p_m = 0$ follows.
\end{proof}

If, in \thmref{modcadlag}, we can choose the modulus of continuity $h_\eps=h\in\CH$, independent of
$\eps$, such that \eqref{eq:1} holds, we do not need to check \eqref{eq:2} and \eqref{eq:3}.

\begin{corollary}[$\eps$-independent modulus of continuity]\label{c:epsindep}
	Assume that\/ $\X^n=(\X^n_t)_{t\in J}$ converges in distribution to an $\MMI$-valued c\`adl\`ag
	process\/ $\X$, and\/ $Q\subseteq J$ is dense.
	Then\/ $\X$ is an\/ $\FMI$-valued c\`adl\`ag process if, for some\/ $h\in \CH$,
	\begin{equation} \label{eq:var1}
		\nlimsup \Ps{\X_t^n \in \Mh} = 1 \qquad \forall t\in Q.
	\end{equation}
\end{corollary}
\begin{proof}
	Let $h\in\CH$ be such that \eqref{eq:var1} is satisfied and set $h_\eps:=h$.
	Then \eqref{eq:3} is satisfied for every choice of $w_\eps \in \CH$, $\eps>0$.
	For every c\`adl\`ag process, in particular for $\X$, there exist \modscadlag\ $w_\eps$ such that \eqref{eq:2} holds (cf.\ \cite[(14.6),(14.8) and (14.46)]{Bil68}).
	Thus, \thmref{modcadlag} yields the claim.
\end{proof}

\section{Examples}\label{sec:examples}

The (neutral) tree-valued Fleming-Viot dynamics is constructed in \cite{GPW13} using the formalism of metric
measure spaces. In \cite{DGP12}, (allelic) types -- encoded as marks of marked metric measure spaces -- are
included, in order to be able to model mutation and selection. 

In \cite[Remark~3.11]{DGP12} and \cite[Theorem~6]{DGP13} it is stated that the resulting tree-valued Fleming-Viot
dynamics with mutation and selection (TFVMS) admits a mark function at all times, almost surely.  The given
proof, however, contains a gap, because it relies on the criterion claimed in \cite[Lemma~7.1]{DGP13}, which is
wrong in general (see \exref{counter}).
The reason why the criterion may fail is a lack of homogeneity of the measure $\nu$, in the sense that
there are parts with high and parts with low mass density. Consequently, if we condition two samples to have
distance less than $\eps$, the probability that they are from the high-density part tends to one as
$\eps\downarrow 0$, and we do not ``see'' the low-density part. This phenomenon occurs if $\nu$ has an
atom but is not purely atomic.
We also give two non-atomic examples, one a subset of Euclidean space, and the other one ultrametric.

\begin{example}[counterexamples]\label{ex:counter}
In both examples, it is straight-forward to see that $(X,r,\mu)$, with $\mu=\nu\otimes K$,
satisfies the assumptions of \cite[Lemma~7.1]{DGP13}, but does not admit a mark function.
The mark space is $I=\{0,1\}$.
\begin{enumerate}
\item Let $\lambda_A$ be Lebesgue measure of appropriate dimension on a set $A$.
	Define $X:=[0,1]^2 \cup [2,3]$, where $[2,3]$ is identified with $[2,3]\times\{0\} \subseteq \R^2$,
	\begin{equation}
		\nu:=\tfrac12({\lambda_{[0,1]^2} + \lambda_{[2,3]}}) \quad\text{and\/}\quad
		K_x:=\begin{cases} \frac12(\delta_0+\delta_1), & x\in [0,1]^2,\\ \delta_0, & x\in [2,3].\end{cases}
	\end{equation}
\item In this example think of a tree consisting of a left part with tertiary branching points and a right part with binary branching points. The leaves correspond to $X:=A\cup B$ with $A=\{0,1,2\}^\N$ and $B=\{3,4\}^\N$, and we choose as a metric
	\begin{equation}
		r\(\folge x, \folge y\) := \max_{n\in\N} e^{-n}\cdot\1_{x_n \ne y_n}.
	\end{equation}
	Note that $(X,r)$ is a compact, ultrametric space. The measure $\nu$ is constructed as follows: choose
	the left respectively right part of the tree with probability $\tfrac12$ each. Going deeper in the tree,
	at each branching point a branch is chosen uniformly. That is, let $\nu_A$ and $\nu_B$ be the Bernoulli
	measures on $A$ and $B$ with uniform marginals on $\{0,1,2\}$ and $\{3,4\}$, respectively. Define
	\begin{equation}
		\nu:=\tfrac12(\nu_A + \nu_B) \quad\text{and\/}\quad
		K_x:=\begin{cases} \frac12(\delta_0+\delta_1), & x\in A,\\ \delta_0, & x\in B.\end{cases}
	\end{equation}
\end{enumerate}
\end{example}

%
\subsection[Tree-valued Fleming-Viot with mutation and selection]{The tree-valued Fleming-Viot dynamics with mutation and selection}\label{sub:FV}
%

In the following, we prove the existence of a mark function for the TFVMS by verifying the assumptions of \thmref{pr-modulus}
for a sequence of approximating tree-valued Moran models.
Due to the Girsanov transform given in \cite[Theorem~2]{DGP12}, it is enough to consider the neutral case, that is
without selection.

We briefly recall the construction of the tree-valued Moran model with mutation (TMMM) with finite population
$U_N=\{1,\ldots,N\}$, $N \in \N$, and types from the mark space $I$. For details and more formal definitions, see
\cite[Subsections~2.1--2.3]{DGP12}. 
In the underlying Moran model with mutation (MMM), every pair of individuals ``resamples'' independently at rate
$\gamma>0$. Here, resampling means that one of the individuals (chosen uniformly at random among the two) is replaced by
an offspring of the other one, and the offspring gets the same type as the parent. Furthermore, every individual
mutates independently at rate $\vartheta\ge 0$, which means that it changes its type according to a fixed
stochastic kernel $\beta(\cdot,\cdot)$ on $I$. Denote the resulting type of individual $x\in U_N$ at time
$t\ge0$ by $\kappa^N_t(x)$.
To obtain the tree-valued dynamics, define the distance $r_t^N(x,y)$ between two individuals $x,y\in U_N$ at
time $t\ge 0$ as twice the time to the most recent common ancestor (MRCA) (cf.\ \cite[(2.7)]{DGP12}), provided
that a common ancestor exists, and as $2t+r_0^N(x,y)$ otherwise. The TMMM is the resulting process
$\X_t^N=(U_N,r_t^N,\nu_N,\kappa_t^N)$, with sampling measure $\nu_N=\tfrac{1}{N} \sum_{k=1}^N \delta_{k}$. It is
easy to check that, by definition, $(U_N, r_t^N)$ is an ultrametric space, provided that the initial metric
space $(U_N, r_0^N)$ is ultrametric. This explains the name \emph{tree-valued} (cf.\ \cite[Remark~2.7]{DGP12}).

Next recall the graphical construction of the MMM from \cite[Definition~2.2]{DGP12}. A resampling event is
modeled by means of a family of independent Poisson point processes $\{ \eta_\mathrm{res}^{k,\ell}: k, \ell \in
U_N\}$ on $\R_+$, where each $\eta_\mathrm{res}^{k,\ell}$ has rate $\gamma/2$. If $t \in
\eta_\mathrm{res}^{k,\ell}$, draw an arrow from $(k,t)$ to $(\ell,t)$ to represent a resampling event at time
$t$, where $\ell$ is an offspring of $k$. Similarly, model mutation times by a family of independent Poisson point
processes $\{ \eta_\mathrm{mut}^k: k \in U_N\}$, where each $\eta_\mathrm{mut}^k$ has rate $\vartheta$.
If $t \in \eta_\mathrm{res}^{k,\ell}$, draw a dot at $(k,t)$ to represent a mutation event changing the type of
individual $k$ (see Figure~\ref{pic:ex_4_1-1}).

Let $(M_t^{t_0,N})_{t \geq t_0}$, $M_t^{t_0,N} \subseteq U_N$ with $M_{t_0}^{t_0,N}=\emptyset$ be the process
that records the individuals of the population at time $t$ with an ancestor at a time $t_0 < s \leq t$ involved
in a mutation event. By a coupling argument, this process can be constructed by means of the Poisson point
processes $(\eta_\mathrm{res}^{k,\ell}, \eta_\mathrm{mut}^k, k,\ell \in U_N)$ as follows (compare
Figures~\ref{pic:ex_4_1-1}--\ref{pic:ex_4_1-2}):
\begin{equation}\label{eq:MtN}
  M_t^{t_0,N} =
  \begin{cases}
   M_{t-}^{t_0,N} \cup \{\ell\} & \mbox{ if there is a resampling arrow from $k \in M_{t-}^{t_0,N}$ to $\ell \in U_N$ at time $t$}, \cr
    M_{t-}^{t_0,N} \cup \{k\} & \mbox{ if there is a mutation event at $k \in U_N$ at time $t$}, \cr
    M_{t-}^{t_0,N} \backslash \{\ell\} & \mbox{ if there is a resampling arrow from $k \notin M_{t-}^{t_0,N}$ to $\ell \in U_N$ at time $t$}. \cr
  \end{cases}
\end{equation}
\picturefig{0.7}{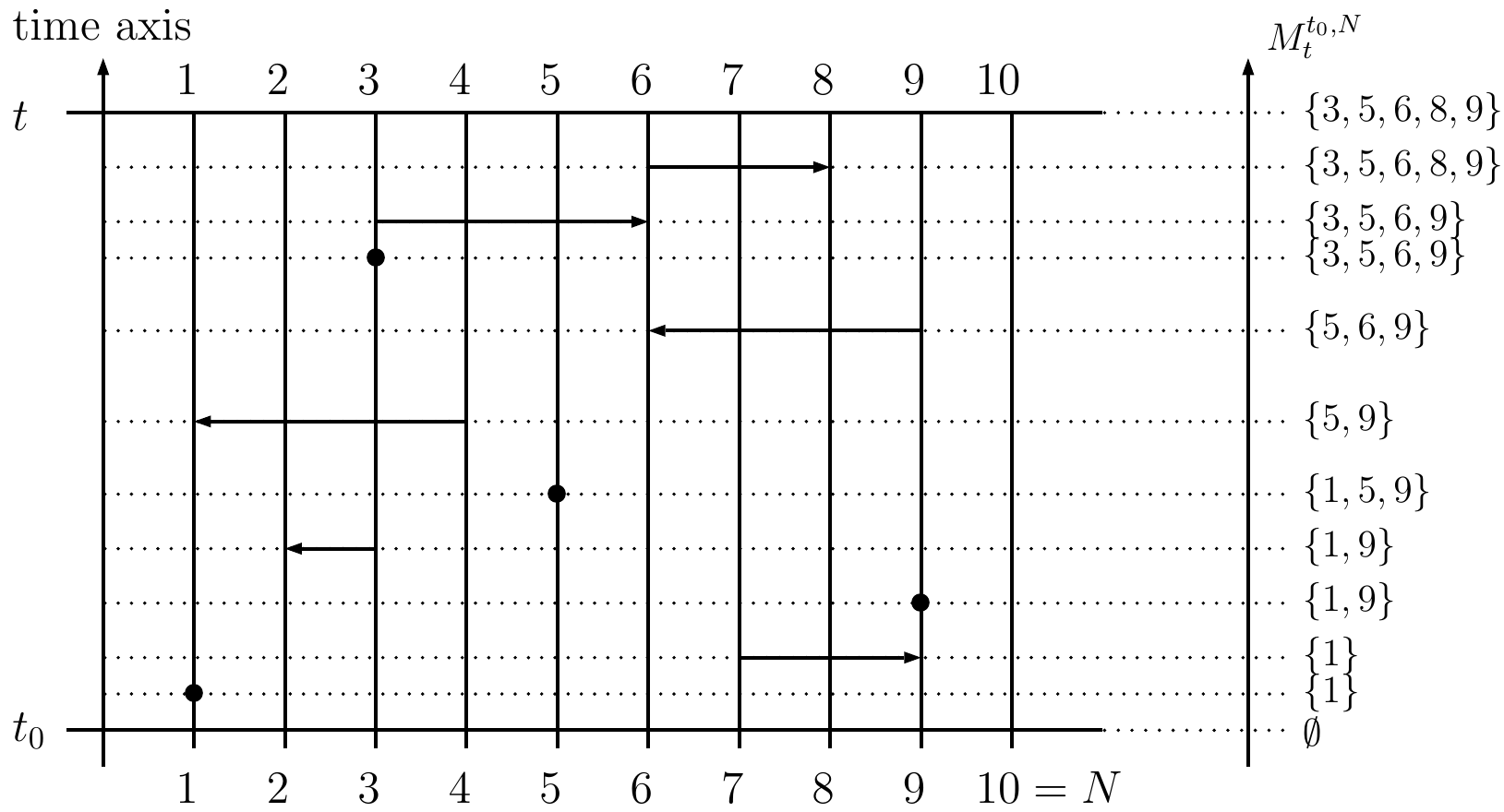}{
Graphical construction of the MMM for $N=10$ for the time-period $[t_0,t]$, and the resulting process
$(M_s^{t_0,N})_{s \in [t_0,t]}$. Resampling arrows are drawn at points of $\eta_\mathrm{res}^{k,\ell}$, and
mutation dots at points of $\eta_\mathrm{mut}^k$.}{ex_4_1-1}
\picturefig{0.7}{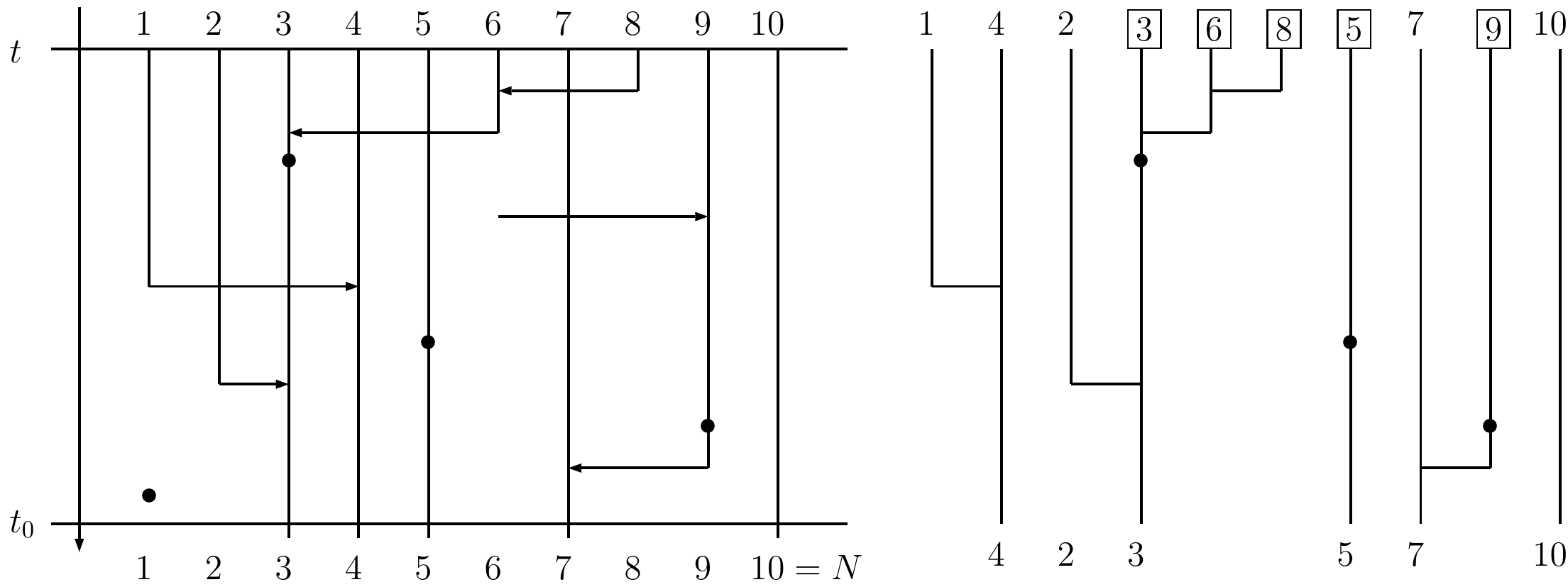}{Tracing the ancestor backwards in time in Figure~\ref{pic:ex_4_1-1}: This dual
construction is also known as the coalescent backwards in time. Reverse the arrows to see for instance that $3$
at time $t_0$ is an ancestor of $8$ at time $t$.
The elements of $M_t^{t_0,N} \subseteq U_N$ are highlighted by boxes in the right part of the picture.}{ex_4_1-2}
Let $\xi_t^N := \frac{1}{N} \#M_{t_0+t}^{t_0,N}$ be the proportion of individuals at time $t_0+t, t \geq 0$ whose ancestors have mutated
after (the for the moment fixed) time $t_0$. 

\begin{lemma}\label{lem:mutbound}
Let\/ $C:=\frac12\vartheta(2\vartheta+\gamma)$. Then for all\/ $a,\delta>0$
\begin{equation} \lbeq{bd-on-Y-exp}
  \limsup_{N \rightarrow \infty} \P\bigl( \sup_{t \in [0,\delta]} \xi_t^N \geq a \bigr)
  \leq C a^{-2} \delta^2.
\end{equation}
\end{lemma}

\begin{proof}
By definition, $\bigl( \xi_t^N \bigr)_{t \geq 0}$ is a (continuous time) Markov jump process on $[0,1]$ with
$\xi_0^N=0$ and transitions
\begin{equation}
  \begin{cases}
    x \mapsto x-1/N & \mbox{ at rate } \frac{\gamma}{2} N^2 x (1-x), \cr
    x \mapsto x+1/N & \mbox{ at rate } \frac{\gamma}{2} N^2 x (1-x) + \vartheta N (1-x). \cr
  \end{cases}
\end{equation}
This process converges weakly with respect to the Skorohod topology to the solution $(Z_t)_{t \geq 0}$ of the
stochastic differential equation (SDE)
\begin{equation}
\lbeq{SDE-mark}
  \d Z_t = \vartheta (1-Z_t) \dt + \sqrt{\gamma Z_t (1-Z_t)}\, \d B_t, \quad Z_{0}=0.
\end{equation}
Indeed, to establish tightness use \cite[Theorem~III.9.4]{EK}. Note that, as $[0,1]$ is compact, it
suffices to show the convergence of the generators applied to a set of appropriate test-functions. For existence
and uniqueness of solutions to \eqref{SDE-mark} reason as for the Bessel SDE in \cite[(48.1) and below]{RW2}.
Moreover, $Z_t \in [0,1]$ is a bounded non-negative right-continuous submartingale. Hence, with Doob's
submartingale inequality (see for instance \cite[Proposition~II.2.16(a)]{EK}), we obtain 
\begin{equation}
  \P\bigl( \sup_{t \in [0,\delta]} Z_t \geq a \bigr) 
  = \P\bigl( \sup_{t \in [0,\delta]} Z_t^2 \geq a^2 \bigr)
  \leq a^{-2} \EE[ Z_{\delta}^2 ].
\end{equation}
As $Z_t \in [0,1]$, we further deduce using It\^o's formula that for all $t \geq 0$,
\begin{align}
  & \EE[ Z_t ] \leq \vartheta t \quad\mbox{ and } \\
  & \EE[ Z_t^2 ] = \EE\bigl[ \int_{0}^t 2 Z_s \vartheta (1-Z_s) + \gamma Z_s (1-Z_s) \,\d s \bigr]
  \leq C t^2.
\end{align}
Then
\begin{equation}
  \limsup_{N \rightarrow \infty} \P\bigl( \sup_{t \in [0,\delta]} \xi_t^N \geq a \bigr)
  \leq \P\bigl( \sup_{t \in [0,\delta]} Z_t \geq a \bigr) 
  \leq  C a^{-2} \delta^2
\end{equation}
follows.
\end{proof}

As the construction of the TFVMS in \cite{DGP12} is only given for a compact type-space $I$, we make the same
assumption. Note, however, that our proof itself does not use compactness and is therefore valid for non-compact
$I$, provided that the TFVMS is the limit of the corresponding Moran models, and there exists a Girsanov
transform allowing us to reduce to the neutral case.

\begin{theorem}[the TFVMS admits a mark-function] \label{t:mark-FV}
Let\/ $I$ be compact and\/ $\X=(\X_t)_{t \geq 0}$ be the tree-valued Fleming-Viot dynamics with mutation and
selection as defined in \textup{\cite{DGP12}}. Then
\begin{equation}
  \P( \X_t \in \FMI \text{ for all\/ } t>0 ) = 1.
\end{equation}
In particular, $(\X_t)_{t>0}$ is an $\FMI$-valued c\`adl\`ag process.
\end{theorem}

\begin{proof}
By \cite[Theorem~2]{DGP12}, there exists a Girsanov transform that enables us to assume without loss of
generality that selection is not present. In this case, according to \cite[Theorem~3]{DGP12}, $\X$ is the limit
in distribution of TMMMs $\X^N=(\X^N_t)_{t\ge0}$, as discussed above. Let $\X^N_t=(U_N, r^N_t, \nu_N, \kappa_t^N)$ with
$U_N=\{1,\ldots,N\}$ and $\nu_N$ the uniform distribution on $U_N$. Let $\delta>0$ be fixed for the moment, and
recall that the distance $r_t^N(x,y)$ between two individuals $x,y\in U_N$ at time $t\ge\delta/2$ is twice the
time to the MRCA. Hence, if $r_t^N(x,y) < \delta$, then $x$ and $y$ at time $t$ have a common ancestor at time
$t-\delta/2$.
Further recall that $(M_t^{t_0,N})_{t \geq t_0}$, with $M_t^{t_0,N} \subseteq U_N$ and
$M_{t_0}^{t_0,N}=\emptyset$, records the individuals of the population at time $t$ with an ancestor at a time
$s\in (t_0, t]$ involved in a mutation event (cf.\ \eqref{eq:MtN}).

Fix an arbitrary time horizon $T>0$ and $i \in \N$, $i \leq 2T/\delta$.
Using the notation of \thmref{pr-modulus}, for $t \in [i \delta / 2,(i+1) \delta / 2)$, let
$Y_{t,\eps,\delta}^N := U_N \backslash M_t^{(i-1)\delta / 2,N}$, independent of $\eps>0$. Set
$Y_{t,\eps,\delta}^N := \emptyset$ for $t<\delta/2$.
We claim that \eqref{eq:modulust} is satisfied for any choice of $h_{t,\eps}\in\CH$.
Indeed, if $x,y \in Y_{t,\eps,\delta}^N$ satisfy $r_t^N(x,y) < \delta$, then they have a common ancestor at time
$t_0:= (i-1)\delta / 2 \le t-\delta/2$, and after this point in time no mutation occurred along their ancestral
lineages. In particular, $d(\kappa_t^N(x),\kappa_t^N(y))=0$, and \eqref{eq:modulust} is obvious.
Moreover, $\X^N$ is $\FMI$-valued by construction, and $\X$ has continuous paths by \cite[Theorem~1]{DGP12}.
According to \thmref{pr-modulus}, it is therefore enough to find moduli of continuity $h_{t,\eps} \in \CH$ such
that \eqref{eq:pasforall} holds for every $\eps>0$.

By \lemref{mutbound}, we obtain a constant $C>0$ such that for every $a>0$,
\begin{equation}
  \limsup_{N \rightarrow \infty} \P\bigl( \sup_{t \in [i \delta / 2,(i+1) \delta / 2)} \nu_N\bigl( U_N
  \setminus Y_{t,\eps,\delta}^N \bigr) \geq a \bigr)
  \leq Ca^{-2}\delta^2.
\end{equation}
After summation over $i \in \{1,\ldots, \floor{2T/\delta}\}$, we obtain
\begin{equation}
\lbeq{mark-bound}
  \limsup_{N \rightarrow \infty} \P\bigl( \sup_{t \in [\delta/2,T]}
  	\nu_N\bigl( U_N \backslash Y_{t,\eps,\delta}^N \bigr) \geq a \bigr)
  \leq 2TC\delta a^{-2}.
\end{equation}
For $\epsilon>0$ arbitrary, we use this inequality with $a := \sqrt{\eps^{-1}2TC\delta}$, together with
$\|\nu_N\|\le 1$ for $t<\delta/2$, to see that \eqref{eq:pasforall} is satisfied for
$h_{t,\eps}\in \CH$ with
\begin{equation}
	h_{t,\eps}(\delta)\ge\sqrt{\eps^{-1}2TC\delta} + \1_{[2t, \infty[}(\delta). \qedhere
\end{equation}
\end{proof}

%
\subsection[Tree-valued $\Lambda$-Fleming-Viot]{The tree-valued $\Lambda$-Fleming-Viot process} \label{sub:TFV}
%

Let $\Lambda$ be a finite measure on $[0,1]$, and recall the $\Lambda$-coalescent, introduced in \cite{Pitman99}.
It is a coalescent process, where each $k$-tuple out of $N$ blocks merges independently at rate
\begin{equation}
	\lambda_{N,k} := \int_0^1 y^{k-2} (1-y)^{N-k}\, \Lambda(\dy).
\end{equation}
For fixed $N$, it is elementary to construct a finite, random (ultra-)metric measure space encoding the random
genealogy of the $\Lambda$-coalescent, where the distance is defined as the time to the MRCA (recall the construction of Figures~\ref{pic:ex_4_1-1}--\ref{pic:ex_4_1-2} and see Figure~\ref{pic:ex_4_1-3}).
\picturefig{0.5}{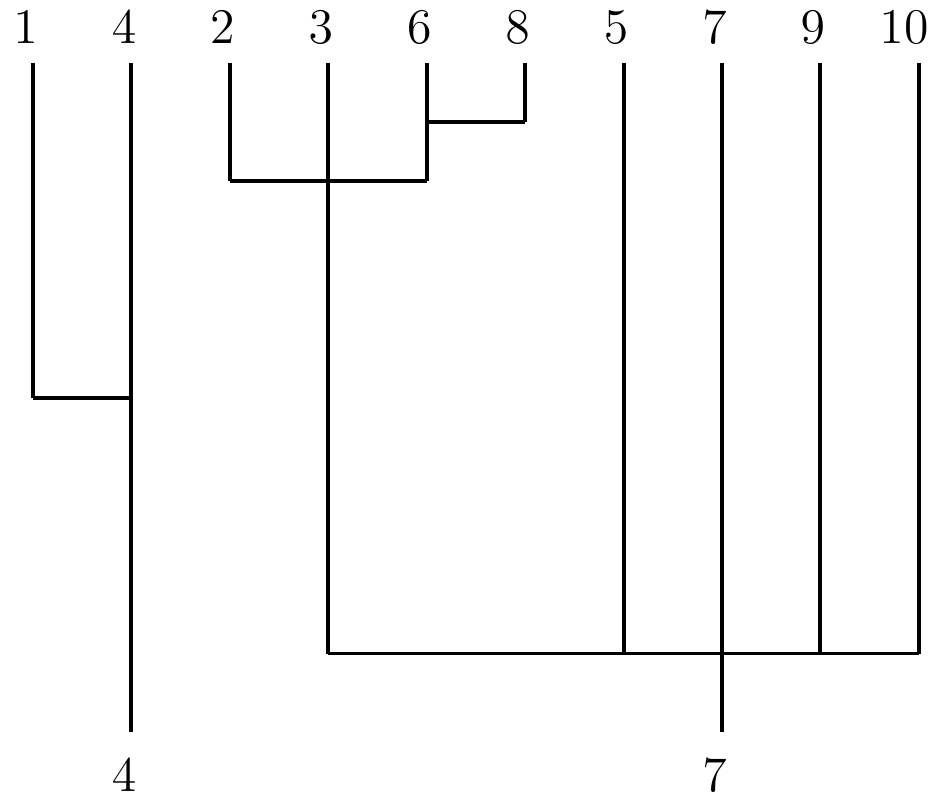}{Tracing the ancestor backwards in time: The $\Lambda$-coalescent allows for one parent to have more than one child.}{ex_4_1-3}
In \cite[Theorem~4]{GPW09}, existence and uniqueness of a Gromov-weak limit in distribution, as $N\to\infty$, is
proven to be equivalent to the so-called ``dust-free"-property, namely $\int_0^1 y^{-1}\, \Lambda(\dy)= \infty$.
The resulting limit is called $\Lambda$-coalescent measure tree.

Now, replace the tree-valued Moran models considered in \subref{FV} and \cite{DGP12} by so-called tree-valued
$\Lambda$-Cannings models with $\Lambda$ satisfying the dust-free-property.
That is, leave the mutation- and selection-part as it is and change the resampling-part of the Moran models as follows:
For $k=2,\ldots,N$, at rate $\binom{N}{k} \lambda_{N,k}$ a block of $k$ individuals is chosen uniformly at
random among the $N$ individuals of the population. Upon such a resampling event, all individuals in this block
are replaced by an offspring of a single individual which is chosen uniformly from this block. Note that the genealogy
(disregarding types) of the resulting $\Lambda$\nobreakdash-Cannings model with $N$ individuals is dual to the
$\Lambda$-coalescent starting with $N$ blocks. We call any limit point (in path space) of the tree-valued
$\Lambda$-Cannings processes, as $N$ tends to infinity and $\Lambda$ is fixed, \emph{tree-valued
$\Lambda$-Fleming-Viot process} (TLFV). In the neutral case, existence and uniqueness of such a limit point
follows as a special case of the forthcoming work \cite{GrevenKlimovskyWinter}.
Here, we show that, whenever limit points exist, all of them admit mark functions.

\begin{theorem}[the TLFV admits a mark-function] \label{t:mark-Lambda}
Suppose there is no selection, that is $\alpha=0$, and\/ $\X=(\X_t)_{t \geq 0}$ is a tree-valued
$\Lambda$-Fleming-Viot process with mutation. Then
\begin{equation}
  \P( \X_t \in \FMI \text{ for all\/ } t>0 ) = 1.
\end{equation}
\end{theorem}

\begin{proof}
By passing to a subsequence if necessary, we may assume that the $\Lambda$-Cannings models converge in
distribution to $\X$.
We proceed as in Subsection~\ref{sub:FV}. Again, let $(M_t^{t_0,N})_{t \geq t_0}$, $M_t^{t_0,N} \subseteq U_N$
with $M_{t_0}^{t_0,N}=\emptyset$ be the process that records the individuals of the population at time $t$ with
an ancestor at a time $t_0 < s \leq t$ involved in a mutation event and  $\xi_t^N := \frac{1}{N}
\#M_{t_0+t}^{t_0,N}$ be the proportion of individuals at time $t_0+t, t \geq 0$ whose ancestors have mutated
after (the for the moment fixed) time $t_0$. By definition, $\bigl( \xi_t^N \bigr)_{t \geq 0}$ is a (continuous
time) Markov jump process on $[0,1]$ with $\xi_0^N=0$ and generator
\begin{align}
  \big(\Omega^N f\big)(x)
  &= \vartheta N (1-x) \big( f(x+1/N)-f(x) \big) \\
  &\phantom{{}={}} + \sum_{k=2}^N \lambda_{N,k}
  	\sum_{m=0}^{(N x) \wedge k} \binom{N x }{m} \binom{N (1-x) }{k-m} \nn\\  
  & \phantom{{}={}+{}} \times \Bigl( \frac{m}{k} \big( f(x+(k-m)/N)-f(x) \big) + \frac{k-m}{k}
  	\big(f(x-m/N)-f(x) \big) \Bigr), \nn
\end{align}
where $x \in [0,1], N \cdot x \in \NN \cup \{0\}, f \in \mathcal{C}_b^2([0,1])$. Due to Taylor's formula, there
is $x_{m,k,N}^+ \in [x,x+(k-m)/N]$, $x_{m,k,N}^- \in [x-m/N,x]$ with
\begin{align}\lbeq{gen-lambda}
  \big(\Omega^N f\big)(x)
  &= \vartheta N (1-x) \big( f(x+1/N)-f(x) \big) \\
  &\phantom{{}={}} + \sum_{k=2}^N \lambda_{N,k}
   \sum_{m=0}^{(N x) \wedge k} \binom{N x }{m} \binom{N (1-x) }{k-m} 
   \Big( \frac{f''(x_{m,k,N}^+)}{2} \frac{m (k-m)^2}{k N^2} +  \frac{f''(x_{m,k,N}^-)}{2} \frac{(k-m)
   m^2}{k N^2} \Big) \nn\\
  &= \vartheta (1-x) f'(x) + O(N^{-1}) + x (1-x) \sum_{k=2}^N \lambda_{N,k} \Delta_{N,k}(x), \nn
\end{align}
where, using $\binom{n }{i} = \frac ni \binom{n-1 }{i-1}$ for $i \geq 1$, 
\begin{equation}
   \Delta_{N,k}(x) = \sum_{m=1}^{(N x) \wedge (k-1)} \binom{Nx-1 }{m-1} \binom{N(1-x)-1 }{k-m-1} 
   \Big( f''(x_{m,k,N}^+) \frac{k-m}{2k} + f''(x_{m,k,N}^-) \frac{m}{2k} \Big).
\end{equation}
Recall that $\sum_{m=0}^k \binom{\ell}m \binom{N-\ell}{k-m} = \binom{N}{k}$ and
$\lambda_{N,k} = \int_0^1 y^{k-2} (1-y)^{N-k}\, \Lambda(\dy)$ with a finite measure $\Lambda$ on $[0,1]$ to see that 
\begin{align}
  \sum_{k=2}^N \left| \Delta_{N,k}(x) \right|
   &\le \|f''\|_\infty \sum_{k=2}^N \lambda_{N,k} \sum_{m=0}^{k-2} \binom{Nx-1}{m} \binom{N(1-x)-1}{k-2-m} \\
   &= \|f''\|_\infty \int_0^1 \sum_{k=2}^N \binom{N-2 }{k-2} y^{k-2} (1-y)^{N-k} \,\Lambda(\dy) \nn\\
   &= \|f''\|_\infty\,\Lambda([0,1]). \nn
\end{align}
Therefore, 
\begin{equation} \lbeq{gen-lambda-bd}
  \big(\Omega^N f\big)(x) = \vartheta (1-x) f'(x) + O(N^{-1}) + x (1-x) O(1).
\end{equation}
Use $f(x)=x, x \in [0,1]$ in \eqref{gen-lambda} to see that $(\xi_t^N)_{t \geq 0}$ is a non-negative
right-continuous submartingale with $\xi_0^N=0$ and $\EE[ \xi_t^N ] \leq \vartheta t$. Use $f(x)=x^2$ to deduce
from \eqref{gen-lambda-bd} that
\begin{equation}
	\mathbb{E}\bigl[ (\xi_t^N)^2 \bigr] \leq C t^2 + O(N^{-1}) t.
\end{equation}
Now reason as for the TFVMS in the proofs of \lemref{mutbound} and \thmref{mark-FV} to complete the claim.
\end{proof}

%
\subsection[Future application: trait-dependent branching]{Future application: Evolving phylogenies of trait-dependent branching} \label{sub:FApp}
%

In \cite{KW2014} the results of the present paper will be applied in a context of evolving genealogies to
establish the existence of a mark function with the help of \thmref{pr-modulus}. These genealogies are random
marked metric measure spaces, constructed as the limit of approximating particle systems. The individual birth-
respectively death-rates in the $N^{\mbox{th}}$-approximating population depend on the present trait of the
individuals alive and are of order $O(N)$. At each birth-event, mutation happens with a fixed probability. Each
individual is assigned mass $1/N$. The metric under consideration is genetic distance: in the
$N^{\mbox{th}}$-approximating population genetic distance is increased by $1/N$ at each birth with mutation.
Hence, genetic distance of two individuals is counted in terms of births with mutation backwards in time to the
MRCA rather than in terms of the time to the MRCA.

Because of the use of exponential times in the modeling of birth- and death-events in this therefore
non-ultrametric setup the analysis of the modulus of continuity of the trait-history of a particle in
combination with the evolution of its genetic age plays a major role in establishing tightness of the
approximating systems and existence of a mark function. In \cite[Lemma~3.9]{K2014}, control on the modulus of
continuity is obtained by transferring the model to the context of historical particle systems. In a first step,
time is related to genetic distance by means of the modulus of continuity. The extend of the change of trait of
an individual in a small amount of time (recall \eqref{eq:pas} and \eqref{eq:modulus2}) can then be controlled
by means of the modulus of continuity of its trait-path in combination with a control on the height of the
largest jump during this period of time. This can in turn be ensured by appropriate assumptions on the mutation
transition kernels of the approximating systems. 


\begin{acknowledgements}
	We are thankful to Anita Winter for discussions in the initial phase of the project, and to the referee
	for helpful comments.
	The research of Sandra Kliem was supported by the DFG through the SPP Priority Programme 1590.
\end{acknowledgements}


\footnotesize\renewcommand{\section}{\subsection}



\begin{thebibliography}{GKW15}

\bibitem[Ald93]{Aldous:CRT3}
David Aldous.
\newblock The continuum random tree {III}.
\newblock {\em Ann. Probab.}, 21(1):248--289, 1993.

\bibitem[Bil68]{Bil68}
Patrick Billingsley.
\newblock {\em Convergence of probability measures}.
\newblock John Wiley \& Sons, Inc., New York-London-Sydney, 1968.

\bibitem[Bog07]{BogachevII}
V.~I. Bogachev.
\newblock {\em Measure Theory, Volume {II}}.
\newblock Springer, 2007.

\bibitem[DGP11]{DGP11}
Andrej Depperschmidt, Andreas Greven, and Peter Pfaffelhuber.
\newblock Marked metric measure spaces.
\newblock {\em Electron. Commun. Probab.}, 16(17):174--188, 2011.

\bibitem[DGP12]{DGP12}
Andrej Depperschmidt, Andreas Greven, and Peter Pfaffelhuber.
\newblock Tree-valued {Fleming-Viot} dynamics with mutation and selection.
\newblock {\em Ann. Appl. Probab.}, 22(6):2560--2615, 2012.

\bibitem[DGP13]{DGP13}
Andrej Depperschmidt, Andreas Greven, and Peter Pfaffelhuber.
\newblock Path-properties of the tree-valued {Fleming-Viot} processes.
\newblock {\em Electron. J. Probab.}, 18(84):1--47, 2013.

\bibitem[EK05]{EK}
S.~N. Ethier and T.~G. Kurtz.
\newblock {\em Markov Processes, characterization and convergence}.
\newblock Wiley, 2005.

\bibitem[GKW15]{GrevenKlimovskyWinter}
Andreas Greven, Anton Klimovsky, and Anita Winter.
\newblock Evolving genealogies for spatial {$\Lambda$-Fleming-Viot} processes
  with mutation.
\newblock In preparation, 2015.

\bibitem[GPW09]{GPW09}
Andreas Greven, Peter Pfaffelhuber, and Anita Winter.
\newblock Convergence in distribution of random metric measure spaces
  ({$\Lambda$}-coalescent measure trees).
\newblock {\em Probab. Theory Related Fields}, 145(1-2):285--322, 2009.

\bibitem[GPW13]{GPW13}
Andreas Greven, Peter Pfaffelhuber, and Anita Winter.
\newblock Tree-valued resampling dynamics. {M}artingale problems and
  applications.
\newblock {\em Probab. Theory Related Fields}, 155:789--838, 2013.

\bibitem[Gro99]{Gromov}
Misha Gromov.
\newblock {\em Metric structures for {R}iemannian and non-{R}iemannian spaces},
  volume 152 of {\em Progress in Mathematics}.
\newblock Birkh\"auser Boston, Inc., Boston, MA, 1999.

\bibitem[Kle14]{Kle14}
Achim Klenke.
\newblock {\em Probability theory}.
\newblock Universitext. Springer, London, second edition, 2014.

\bibitem[Kli14]{K2014}
Sandra Kliem.
\newblock A compact containment result for nonlinear historical superprocess
  approximations for population models with trait-dependence.
\newblock {\em Electron. J. Probab.}, 19(97):1--13, 2014.

\bibitem[KW15]{KW2014}
Sandra Kliem and Anita Winter.
\newblock Evolving phylogenies of trait-dependent branching with mutation and
  competition.
\newblock In preparation, 2015.

\bibitem[L{\"o}h13]{Loehr13}
Wolfgang L{\"o}hr.
\newblock Equivalence of {G}romov-{P}rohorov- and {G}romov's
  {$\underline\square_\lambda$}-metric on the space of metric measure spaces.
\newblock {\em Electron. Commun. Probab.}, 18(17):1--10, 2013.

\bibitem[LVW14]{LoehrVoisinWinter14}
Wolfgang L{\"o}hr, Guillaume Voisin, and Anita Winter.
\newblock Convergence of bi-measure {$\mathbb R$-trees} and the pruning
  process.
\newblock {\em Ann. Inst. H. Poincar\'e Probab. Statist.}, in press:36 pages,
  2014.
\newblock arXiv:1304.6035.

\bibitem[Pio11]{Piotrowiak:phd}
Sven Piotrowiak.
\newblock {\em Dynamics of Genealogical Trees for Type- and State-dependent
  Resampling Models}.
\newblock PhD thesis, University of Erlangen-Nuremberg, 2011.

\bibitem[Pit99]{Pitman99}
Jim Pitman.
\newblock Coalescents with multiple collisions.
\newblock {\em Ann. Probab.}, 27(4):1870--1902, 1999.

\bibitem[RW00]{RW2}
L.~C.~G. Rogers and D.~Williams.
\newblock {\em Diffusions, Markov Processes and Martingales, Volume {II}}.
\newblock Cambridge University Press, second edition, 2000.

\end{thebibliography}
\end{document}